\documentclass[12pt]{amsart}

\usepackage{amsmath, amssymb, amsfonts, amsthm}
\usepackage{mathrsfs, bm, upgreek}
\usepackage{enumitem}
\usepackage{url}
\usepackage{booktabs}
\usepackage{tikz}
\usepackage{tikz-cd}
\usepackage{todonotes}
\usepackage[normalem]{ulem}
\usepackage{marginnote}

\usepackage{float}
\newcommand{\brho}{{\bm{\uprho}}}

\newcommand{\R}{\mathbb{R}}
\newcommand{\C}{\mathbb{C}}
\newcommand{\F}{\mathbb{F}}
\newcommand{\OK}{\mathscr{O}_K}
\newcommand{\OKS}{\mathscr{O}_{K,S}}
\newcommand{\Z}{\mathbb{Z}}
\newcommand{\Q}{\mathbb{Q}}
\newcommand{\Lat}{\mathscr{L}}
\newcommand{\bx}{\mathbf{x}}
\newcommand{\by}{\mathbf{y}}

\newcommand{\fp}{\mathfrak{p}}

\newcommand{\RNb}{500}

\DeclareMathOperator{\rfv}{rfv}
\DeclareMathOperator{\lcm}{lcm}
\DeclareMathOperator{\Disc}{Disc}
\DeclareMathOperator{\ord}{ord}

\numberwithin{table}{section}

\theoremstyle{plain}
\newtheorem{theorem}{Theorem}[section]
\newtheorem{lemma}[theorem]{Lemma}
\newtheorem{corollary}[theorem]{Corollary}

\theoremstyle{definition}
\newtheorem{definition}{Definition}[section]
\newtheorem{problem}[definition]{Problem}
\newtheorem{remark}[definition]{Remark}
\newtheorem{algorithm}[definition]{Algorithm}

\title[Solving $S$-unit equations]{A robust implementation for solving the $S$-unit equation and several applications}

\author[Alvarado et al.]{Alejandra Alvarado}
\address{Alejandra Alvarado, Department of Mathematics and Computer Science, Eastern Illinois University}
\email{aalvarado2@eiu.edu }
\author[]{Angelos Koutsianas}
\address{Angelos Koutisanas, Department of Mathematics, University of British Columbia}
\email{koutsis.jr@gmail.com}
\author[]{Beth Malmskog}
\address{Beth Malmskog, Department of Mathematics and Computer Science, Colorado College}
\email{beth.malmskog@gmail.com}
\author[]{Christopher Rasmussen}
\address{Christopher Rasmussen, Department of Mathematics and Computer Science, Wesleyan University}
\email{crasmussen@wesleyan.edu}
\author[]{Christelle Vincent}
\address{Christelle Vincent, Department of Mathematics and Statistics, University of Vermont}
\email{christelle.vincent@uvm.edu}
\author[]{Mckenzie West}
\address{Mckenzie West, Department of Mathematics, University of Wisconsin Eau Claire}
\email{WestMR@uwec.edu}

\begin{document}

\begin{abstract}
Let $K$ be a number field, and $S$ a finite set of places in $K$ containing all infinite places. We present an implementation for solving the $S$-unit equation $x + y = 1$, $x,y \in \OKS^\times$ in the computer algebra package SageMath. This paper outlines the mathematical basis for the implementation. We discuss and reference the results of extensive computations, including exponent bounds for solutions in many fields of small degree for small sets $S$. As an application, we prove an asymptotic version of Fermat's Last Theorem for totally real cubic number fields with bounded discriminant where 2 is totally ramified.  In addition, we use the implementation to find all solutions to some cubic Ramanujan-Nagell equations.
\end{abstract}

\maketitle

\section{Introduction}\label{sec:intro}

In 1909, Thue proved there are only finitely many integral solutions to what we now call the Thue equation; i.e, that for any $\mathbb{Q}$-irreducible binary form $F(X,Y)$ of degree at least 3, defined over the integers, there are only finitely many solutions $(x,y)\in\mathbb{Z}^2$ to the equation \[F(x,y)=c,\] where $c$ is any non-zero integer \cite{Thue:1909}. Thue accomplished this by formally factoring $F$ into linear terms of the form $(x-\alpha y)$, where $\alpha$ is algebraic, then bounding the quality of rational approximations of $\alpha$ in terms of the size of $x$ and $y$.  Thus bounds on integer solutions to the Thue equation arose out of the theory of approximating algebraic numbers by rationals. Thue's theorem was generalized by Siegel \cite{Siegel:1929}\footnote{See also the recent translation \cite{Fuchs:2014} by Fuchs.} and then Mahler \cite{Mahler:1933}.  These generalizations gave rise to a central fact of modern computational number theory: if $K$ is a number field, and $S$ a finite list of places of $K$ including all infinite places, then there are only finitely many solutions $(x,y)$ to the equation
\begin{equation}\label{introeqn}
x + y = 1, \qquad x, y \in \OKS^\times.
\end{equation}
Here, $\OKS^\times$ is the unit group of the ring $\OKS$ of $S$-integers in $K$. We refer to \eqref{introeqn} as the \emph{$S$-unit equation}. In this paper, we describe an algorithm to determine the complete set of solutions to the $S$-unit equation for general $K$ and $S$. More generally, for fixed $a,b \in \OKS$, we can see that the equation $ax + by = 1$ will also have only finitely many solutions by expanding the set $S$ to include all primes dividing $a$ and $b$ and searching for solutions to  \eqref{introeqn}. Thus it suffices to solve \eqref{introeqn} to address the more general case, and we focus on \eqref{introeqn} here (though it should be remarked that this is not the most efficient way to solve $ax + by = 1$).

The work of Gelfond and Schneider, resolving Hilbert's seventh problem in the affirmative (all irrational algebraic powers of algebraic numbers are transcendental once trivial cases are ignored), determined lower bounds on the absolute value of a $\mathbb{Q}$-linear combination of  two $\mathbb{Q}$-linearly independent logarithms of algebraic numbers.  Alan Baker's 1967 theorem \cite{Baker:1967} generalized these results to the case of many logarithms. Baker, W\"{u}stholz, and many others continued to improve these bounds. Naturally, one should ask if similar results are available over local fields, and indeed such results began to appear quickly. In 1968, Brumer proved the first analogue of Baker's work for $p$-adic logarithms \cite{Brumer:1967}, followed by many improvements and generalizations, such as the results of Yu \cite{Yu:1989}. Improvements in both the archimedean and nonarchimedean cases continue to appear, such as in \cite{GyoryYu:2006, Baker-Wustholz:2007, Yu:2007, Gyory:2019}.

For any choice of $K$ and $S$, $\OKS^\times$ is a finitely generated $\Z$-module. Fixing a basis $\rho_1, \dots, \rho_t$ for the torsion free part, we can express any $x \in \OKS^\times$ as $x = \xi \cdot \prod_{i=1}^t \rho_i^{a_i}$ for some root of unity $\xi \in K$ and some $a_i \in \Z$. Building on the lower bounds for linear combinations of logarithms, Gy\H{o}ry \cite{Gyory:1979} determined effectively computable bounds for the exponents $a_i$. This was a great victory for computational number theory, as this provably restricted all solutions to \eqref{introeqn} to a finite search space. Unfortunately, the demonstrated bounds were enormous and as a matter of practice, it was computationally infeasible to conduct an exhaustive search for solutions, even in the very simplest cases. Baker and Davenport devised a clever method of reducing the bounds in special cases in \cite{BakerDavenport:1968}. However, in \cite{deWeger:1989}, de Weger built on the ideas of Baker-Davenport to develop a powerful general method of algorithmically reducing the bounds to a manageable size, relying on the lattice basis reduction algorithm of Lenstra, Lenstra, and Lov\'{a}sz \cite{LLL:1982} (henceforth referred to as the ``LLL algorithm''). Though it is has not been proven that de Weger's method will always reduce the bounds coming from the results in linear forms of logarithms, this is the rule in practice. In many cases, de Weger's approach provides sufficient improvements that, with careful sieving (or sometimes even with only brute force), the entire search space can be  exhausted and complete lists of solutions can be enumerated.

Beyond the improvements provided by LLL-based reduction, many mathematicians have developed further algorithms for efficiently searching below the ``LLL bounds'' provided by de Weger's work. Two powerful examples are reported in \cite{Wildanger:2000} and \cite{Smart:1999}. Increasingly, the theoretical improvements (assisted by technological improvements) have pushed ambitious and interesting computational problems within reach. For example, Smart determined the entire set of all genus $2$ curves over $\Q$ with good reduction away from $2$, based in part on solving \eqref{introeqn} for a family of number fields unramified away from $2$ \cite{Smart:1997}.

We have written a package of Python functions for inclusion in the computer algebra system SageMath \cite{SAGE}, which solves the $S$-unit equation \eqref{introeqn} over any number field $K$ and for any finite set $S$ of finite places. As experienced readers may expect, the package is not practical when either $[K:\Q]$ or $|S|$ is too large, although there is no theoretical obstruction. While this package is the independent creation of the authors, it is based in part on the descriptions of algorithms implemented by Smart \cite{Smart:1995, Smart:1997, Smart:1998}. Specifically, we follow Smart's development in determining initial large bounds, including the numbering of constants, in \cite{Smart:1995}, with some adjustments and small corrections. In reducing the bounds, we follow \cite{Smart:1998}, again with some adjustments. The sieving step is based on ideas cited by Smart \cite{Smart:1997} as due to others (as noted in Section \ref{sec:sieve}) but has been redeveloped in new notation and style. We include proofs of our versions of results when we made adjustments to versions in the literature. To the authors' knowledge, our package is the first publicly available implementation for solving the $S$-unit equation over any field other than $\Q$; the present article describes the algorithm and its implementation.  The implementation was a highly non-trivial undertaking, involving efforts spreading over more than seven years on the parts of individuals and the entire team.

We also provide new results facilitated by our implementation. In particular, we first provide a discussion of and link to explicit exponent bounds for solutions of the $S$-unit equation in all cases $(K,S)$ where $K/\mathbb{Q}$ is ramified only at primes above some subset of $\{2, 3\}$ and
\[ [K:\Q] \leq 5, \qquad S \subseteq \{ \mathfrak{p} \subseteq \OK : \mathfrak{p} \mid 6 \}. \] We improve the best known exponent bounds for solutions of the $S$-unit equation over number fields related to a class of genus $2$ curves over $\Q$ with good reduction away from $3$. We solve the $S$-unit equation in the $13$ totally real cubic number fields $K$ in which $2$ is totally ramified and the absolute discriminant of $K$, $\Delta_K$, satisfies $|\Delta_K| \leq 2000$, and we use these results to verify that an asymptotic version of Fermat's Last Theorem holds over these fields. Finally, we find all solutions to certain cubic Ramanujan-Nagell equations.

\subsection{Overview}

The organization of the paper proceeds as follows. We introduce certain notations in \S\ref{sec:notation}. In \S\ref{sec:boundintro}, we review the relevant work of Baker-W\"ustholz and Yu. This is used in \S\ref{sec:initial_bound} to establish a ``pre-LLL'' exponent bound for each place in $S$. In \S\ref{sec:LLL_red}, we explain the process of using LLL to reduce these exponent bounds -- the approach is different for archimedean and nonarchimedean places. In \S\ref{sec:sieve}, we describe the sieve for further constraining the final search space. We devote \S\ref{sec:observ} to a discussion of our experimental observations, having now executed our algorithm in several dozen cases. We highlight a special condition ($S$ contains only one finite place) under which a significant improvement in the search space can be obtained. Although narrow in scope, the special condition is sufficiently natural, and the savings sufficiently nontrivial, as to warrant its discussion.
Finally, \S\ref{sec:applications} introduces two applications: an asymptotic version of Fermat's Last Theorem over totally real cubic fields and a solution to a cubic variant of the Ramanujan-Nagell equation.

\subsection*{Acknowledgments}

We are delighted to recognize the Institute for Computational and Experimental Research in Mathematics for both funding and hosting a 2017 collaboration during which a great deal of this project was completed. Part of this work began at the 2014 workshop SageDays 62, and we would like to thank Anna Haensch and Lola Thompson for organizing that workshop and Microsoft Research and The Beatrice Yormark Fund for Women in Mathematics for funding. Some of the work was supported by the van Vleck fund at Wesleyan University. The authors would like to thank many people for helpful conversations that led to improvements in the code and gave direction to this project, including Bjorn Poonen, Andrew Sutherland, and Norman Danner.  We would also like especially to thank David Roe for his contributions to refining and reviewing the code for inclusion in SageMath.  The third author was partially supported in this work by NSA Grant \#H98230-16-1-0300. We are very grateful to the anonymous referees for their careful reading of this work and their many helpful comments which have improved the quality of this paper.

\section{Notation}\label{sec:notation}

\subsection{$S$-units in number fields} Throughout this paper, we let $\bar{\Q}$ denote the algebraic closure of $\Q$ inside $\C$, the field of complex numbers. Unless stated otherwise, we fix the following notation throughout:
\begin{center}
  \begin{tabular}{rcp{0.72\textwidth}}
    $K$                   &  & a number field (assumed to be a subfield of $\bar{\Q}$), \\
    $d_K$                 &  & the absolute degree $[K:\Q]$ \\
    $w$                   &  & the number of distinct roots of unity in $K$ \\
    $\Delta_K$            &  & the absolute discriminant of $K/\Q$ \\
    $\OK$                 &  & the ring of integers of $K$ \\
    $e_\mathfrak{p}$ & & the ramification index of $\mathfrak{p}$ in $K/\Q$ \\
    $f_\mathfrak{p}$      &  & the inertial degree of the prime $\mathfrak{p} \subseteq \OK$ over the rational prime $\mathfrak{p} \cap \Z$ \\
    $r$                   &  & the rank of $\OK^\times$ as a $\Z$-module \\
    $S_\mathrm{fin}$      &  & a set $\{\mathfrak{p}_1, \dots, \mathfrak{p}_s\}$ of $s$ finite places of $K$ \\
    $S_\infty$            &  & the set $\{ \mathfrak{p}_{s+1}, \dots, \mathfrak{p}_{r+s+1} \}$ of all infinite places of $K$ \\
    $S$                   &  & $S_\mathrm{fin} \cup S_\infty = \{\mathfrak{p}_1, \dots, \mathfrak{p}_{r+s+1} \}$ \\
    $S_\Q$                &  & the set of places of $\Q$ which extend to places of $K$ in $S$ \\
    $\OKS$                &  & the ring of $S$-integers in $K$ \\
    $\OKS^\times$         &  & the group of $S$-units in $K$ \\
    $t$                   &  & the rank of $\OKS^\times$ as a $\Z$-module (so $t = r + s$) \\
    $\rho_0$              &  & a root of unity generating the torsion part of $\OKS^\times$ \\
    $\rho_1,\dots,\rho_t$ &  & an ordered basis for the torsion-free part of the $\Z$-module $\OKS^\times$ \\
    $\brho$               &  & the ordered list $[ \rho_0, \rho_1, \dots, \rho_t ]$ \\
  \end{tabular}
\end{center}
 If $f(x) \in \Z[x]$ is a monic and irreducible polynomial, we let $K_f$ denote the number field $\Q(\xi)$, where $\xi$ is a root of $f(x)$. Always, $\log$ denotes the principal branch of the complex logarithm function, with argument in $(-\pi, \pi]$.

\subsection{Absolute Values and Completions}\label{sec:abs_vals}
Each place of $K$ determines an associated value, $| \cdot |_\mathfrak{p}$, which we now describe.

Let $|\cdot|$ denote the usual absolute value on $\C$. If $\mathfrak{p}$ is an infinite place, choose $\sigma_\mathfrak{p} \colon K \to \C$, an embedding corresponding to $\mathfrak{p}$. The associated absolute value depends on whether $\mathfrak{p}$ is a real or complex (meaning non-real) place of $K$:
\[ \left| \alpha \right|_{\mathfrak{p}} := \begin{cases} \left| \sigma_\mathfrak{p}(\alpha) \right| & \mathfrak{p}\ \text{is real}, \\ \left| \sigma_\mathfrak{p}(\alpha) \right| & \mathfrak{p}\ \text{is complex.} \end{cases} \]

Now suppose $\mathfrak{p}$ is a finite place. View $\mathfrak{p}$ as a prime ideal of $\OK$, and let $p$ be the characteristic of the residue field $\OK/\mathfrak{p}$. Let $\ord_\mathfrak{p}$ denote the ordinal function for $\mathfrak{p}$. On $\OK^\times$ this is defined by
\[ \ord_\mathfrak{p}(\beta) = m\ \quad \text{if}\ \beta \in \mathfrak{p}^m - \mathfrak{p}^{m+1}, \]
and it extends to $K^\times$ in the obvious way. We let $|\cdot|_p$ denote the usual absolute value of the $p$-adic field $\Q_p$. The absolute value associated to $\mathfrak{p}$ on $K$ is
\[ \left| \alpha \right|_\mathfrak{p} := p^{-f_\mathfrak{p} \ord_\mathfrak{p}(\alpha)}. \]
Let $K_\mathfrak{p}$ be the $\mathfrak{p}$-adic completion of $K$ with respect to $|\cdot|_\mathfrak{p}$; we also use $|\cdot|_\mathfrak{p}$ for the absolute value on $K_\mathfrak{p}$.

We fix once and for all an algebraic closure $\bar{\Q}_p$ of $\Q_p$, and let $\C_p$ denote the completion of $\bar{\Q}_p$. We use $|\cdot|_p$ to denote the natural extension of $|\cdot|_p$ to all of $\C_p$. We define $\ord_p$ on $\C_p^\times$ to satisfy
\[ |\alpha|_p = p^{-\ord_p \alpha}, \qquad \alpha \in \C_p^\times. \]
As $p\OK$ may split into several prime ideals, the absolute value $|\cdot|_p$ on $\Q$ may have several inequivalent extesnions to $K$, of which $|\cdot|_\mathfrak{p}$ is just one; so we must take care when viewing $K_\mathfrak{p}$ as a subfield of $\bar{\Q}_p$.

For any embedding $\vartheta \colon K \to \bar{\Q}_p$, we obtain a subfield of $\bar{\Q}_p$ as the composite $\vartheta(K) \cdot \Q_p$. By the Prolongation Theorem \cite[\S18.5]{Hasse:1980}, there exists a choice of $\vartheta$ such that ${(K_\mathfrak{p}, |\cdot|_\mathfrak{p})}$ is value-isomorphic to ${(\vartheta(K)\Q_p, |\cdot|_p)}$. Henceforth, we always use this isomorphism to view $K_\mathfrak{p}$ as a subfield of $\bar{\Q}_p$. As the isomorphism respects the valuations, we know $\ord_p$ and $\ord_\mathfrak{p}$ satisfy
\begin{equation}\label{eq:ord_p_ord_frak_p}
\ord_{\mathfrak{p}} \beta = e_{\mathfrak{p}} \ord_p \beta, \qquad \beta \in K_\mathfrak{p}.
\end{equation}

\subsection{Height functions}\label{sec:ht_funcs}
Suppose $n \geq 1$. We let $h$ denote the standard logarithmic Weil height on $\mathbb{P}^n(K)$. This is defined as follows: for any ${\mathbf{x} = (x_0 : \cdots : x_n)} \in \mathbb{P}^n(K)$,
\[ h(\mathbf{x}) = \frac{1}{d_K}\sum_\mathfrak{p} \log\, (\max_j \,\{ |x_j|_{\mathfrak{p}} \} ), \]
where the sum runs over all places of $K$. It is a consequence of the product formula (\cite[Ch.~20, pgs.~326--327]{Hasse:1980}) that $h(\mathbf{x})$ is independent of the choice of coordinates for $\mathbf{x}$.  For any $\alpha \in K$, set ${h(\alpha)=h((1 : \alpha))}$.  Note that this height is \emph{absolute} in the sense that it is not dependent on which field extension $K$ containing the coordinates of $\mathbf{x}$ is considered.

We introduce a modified version of this height function, used in \S\ref{sec:boundintro}. Suppose $\alpha_1, \dots, \alpha_n \in K$, and let $K' = \Q(\alpha_1, \dots, \alpha_n) \subseteq K$. For any nonzero element $\beta \in K'$, we define the function $h'$ by
\[ h'(\beta) = \frac{1}{d_{K'}} \max \left\{ d_{K'} \cdot h(\beta), \left| \log \beta \right|, 1 \right\}. \]
The definition of another height function, $h_\mathfrak{p}$, is slightly more technical and will be introduced when needed in \S\ref{sec:boundintro}.

\subsection{$p$-adic logarithms}\label{section:plog}
Inside $\C_p$, consider the open disk
\[ \Delta_1 := \{z \in \C_p : |z - 1|_p < 1 \}. \]
On $\Delta_1$, we define the $p$-adic logarithm by the series
\begin{equation}\label{eqn:log_series}
\log_p z = -\sum_{n\geq 1}\frac{(1-z)^n}{n}.
\end{equation}
The series is convergent on $\Delta_1$; moreover, on $\Delta_1$ it satisfies the identity
\begin{equation}\label{eqn:logs}
  \log_p(xy)=\log_p x + \log_p y.
\end{equation}
If $|z|_p < p^{-\frac{1}{p-1}}$ we have
\begin{equation}\label{eq:ord_logp}
\ord_p \left( \log_p(1+z) \right) = \ord_p z.
\end{equation}
Based on an idea due to Iwasawa, the $p$-adic logarithm can be extended to any $z \in \C_p$ such that $|z|_p=1$; this extension continues to satisfy \eqref{eqn:logs} (see \cite[II.2.4]{Smart:1998}).

\subsection{Solutions to the $S$-unit equation}\label{sec:sols_Sunit}
We let $A_{K,S}$ denote the additive $\Z$-module $(\Z/w\Z) \times \Z^t$. This is isomorphic to $\OKS^\times$, and the list of generators $\brho$ determines an isomorphism
\[ \Phi_\brho \colon A_{K,S} \longrightarrow \OKS^\times, \qquad \mathbf{a} := (a_0, a_1, \dots, a_t) \mapsto  \prod_{i=0}^t \rho_i^{a_i}. \]
We use the shorthand $\brho^\mathbf{a} := \Phi_\brho(\mathbf{a})$. For obvious reasons, we call the elements of $A_{K,S}$ \emph{exponent vectors}. Much of our discussion will focus on bounds for the entries of an exponent vector. For $\mathbf{a} \in A_{K,S}$, we use the notation $|\mathbf{a}| \leq B$ to signify
\[ \max_{0 \,<\, i \,\leq\, t} \; \left| a_i \right| \leq B. \]
Within $\OKS^\times$, we wish to determine
\[ X_{K,S} := \{\tau \in \OKS^\times : 1 - \tau \in \OKS^\times \}. \]
Solving the $S$-unit equation is equivalent to determining the set $X_{K,S}$. We let $E_{K,S}$ denote the corresponding subset $\Phi_\brho^{-1}(X_{K,S})$ of $A_{K,S}$.

\section{The Bounds of Baker-W\"ustholz and Yu}\label{sec:boundintro}

Suppose $\tau_1, \tau_2 \in \OKS^{\times}$ provide a solution to the $S$-unit equation, so that $\tau_1 + \tau_2 = 1$. With respect to the ordered generating set $\brho$, there are unique vectors $\mathbf{b}_i = (b_{i,0}, \dots, b_{i,t}) \in A_{K,S}$ such that
\begin{equation}\label{eqn:tau_i}
\tau_i = \brho^{\mathbf{b}_i} = \prod_{j=0}^t \rho_j^{b_{i,j}}, \qquad i = 1, 2.
\end{equation}
The techniques of lattice reduction discussed in \S\ref{sec:LLL_red} will not produce an absolute bound for $\left| b_{i,j} \right|$ on their own; they can only be used to improve a known bound. So in this section, we recall bounds established by Baker-W\"ustholz \cite{Baker-Wustholz:1993} and Kunrui Yu \cite{Yu:2007}.
An excellent treatment of the background material appears in \cite{Evertse-Gyory:2015}.

\subsection{Statement of Yu's Bound}
Let $\mathfrak{p}$ be a finite place of $K$, and let $p$ denote the rational prime below $\mathfrak{p}$. We let $q$ be the smallest rational prime distinct from $p$ (so $q = 2$ unless $p = 2$, in which case $q = 3$).
Let $\zeta_m := \exp(2 \pi i / m)$. We say $K$ \emph{satisifies Yu's auxiliary condition} if any of the following hold:
\begin{enumerate}[label={(\roman*)}]
  \item $q = 2$ and $p^{f_\mathfrak{p}} \equiv 1 \bmod{4}$,
  \item $q = 2$ and $\zeta_4 \in K$,
  \item $q = 3$ and $\zeta_3 \in K$.
\end{enumerate}
At the end of this section, we explain how the algorithm finds a bound in cases where $K$ does not satisfy Yu's auxiliary condition.
\begin{theorem}[Yu, {\cite[pg.~190]{Yu:2007}}]\label{thm:Yu_2007_bound}
Suppose $n \geq 1$ and $\mathfrak{p}$ is a prime of $\OK$. Suppose $K$ is a number field satisfying Yu's auxiliary condition and $\mu_0, \mu_1, \dots, \mu_{n-1} \in K^\times$ are chosen which satisfy
\begin{equation}\label{eqn:ord_p_is_zero}
  \ord_\mathfrak{p} \mu_j = 0, \qquad 0 \leq j \leq n-1.
\end{equation}
Suppose $b_j \in \Z$ and $\Theta := \prod\limits_{j=0}^{n-1} \mu_j^{b_j} \neq 1$. Finally, suppose $B$ satisfies
\[ B \geq \max \{ |b_0|, \dots, |b_{n-1}|, 3 \}. \]
Then there exist explicit constants $C_1^*$ and $\Omega$, given below, such that
\[ \ord_\mathfrak{p} \left( \Theta - 1 \right) < C_1^* \Omega \log B. \]
\end{theorem}

\subsection{The constants $\Omega$ and $\Omega'$}

We first discuss the constant $\Omega$, and the variant $\Omega'$ used in the algorithm. In Theorem \ref{thm:Yu_2007_bound}, $\Omega$ is roughly a product of the logarithmic heights of the $\mu_j$. More precisely, decompose the set $\{\mu_j\}_j$ into a disjoint union $\mathfrak{a} \cup \mathfrak{b}$, where $\mathfrak{a}$ is a maximal subset of $\{\mu_j\}_j$ which is multiplicatively independent. Such a decomposition need not be unique. Because of the possible dependence among the $\mu_j$, Yu requires a modified height function:
\[ h_\mathfrak{p}(\mu) := \max \left\{ h(\mu), \frac{ f_\mathfrak{p} }{\kappa_1 (n + 4) d_K} \right\}. \]
(The value $\kappa_1$ is explained in the following subsection.) The constant $\Omega$, which depends on $n$, $d_K$, $\mathfrak{p}$ as well as the $\mu_j$, is then defined by
\[ \Omega(n, d_K, \mathfrak{p}) := \prod_{\mu \in \mathfrak{a}} h(\mu) \cdot \prod_{\mu \in \mathfrak{b}} h_\mathfrak{p}(\mu). \]
As shown in \cite{Yu:2007}, one may choose any maximal independent set $\mathfrak{a}$ for the computation of $\Omega$. If optimization of the bound is critical, one may search over all possible $\mathfrak{a}$ and take the smallest possible bound. This observation is moot in our use, however.
\begin{corollary}\label{corollary:Yu_bound}
Keeping the hypotheses of the previous theorem, suppose also that $\mu_1,\dots,\mu_{n-1}$ are multiplicatively independent. Set
\[ \Omega'(n, d_K, \mathfrak{p}) := h_\mathfrak{p}(\mu_0) \prod_{j=1}^{n-1} h(\mu_j). \]
Then $\displaystyle \ord_\mathfrak{p} \left( \Theta - 1 \right) < C_1^* \Omega' \log B$.
\end{corollary}
\begin{proof}
In this case, $\mathfrak{b}$ is unique; either $\mathfrak{b} = \{\mu_0\}$ or $\mathfrak{b} = \varnothing$. In either case, $\Omega \leq \Omega'$ and the result follows immediately.
\end{proof}
In the algorithm, we are always in the situation of the Corollary. Rather than decide the question of independence between $\mu_0$ and the other $\mu_j$, we just  use the constant $\Omega'$.

\subsection{The constant $C_1^*$}
The value of $C_1^* := C_1^*(n,d_K,\mathfrak{p})$ is dependent on $n$, $d_K$, and $\mathfrak{p}$, as follows. Let $u := \ord_q w$, so that $q^u$ is the $q$-part of $w$. Set
\begin{align*}
  k_2 & := c^{(1)} a^{(1)} \cdot \frac{ n^n\cdot(n+1)^{n+1} }{n!}, \\
  k_3 & := \frac{ p^{f_\mathfrak{p}} }{ q^u } \left( \frac{d_K}{f_\mathfrak{p} \log p} \right)^{n+2} \cdot \log \max \{d_K, e \}, \\
  k_4 & := \max \left\{ \log \left( e^4 (n + 1) d_K \right), e_\mathfrak{p}, f_\mathfrak{p} \log p \right\}.
\end{align*}
Here, $e$ denotes the base of the natural logarithm. The constants $a^{(1)}$, $\kappa_1$, and $c^{(1)}$ are given in Tables \ref{table:a1_kappa1} and \ref{table:c1}. Finally,
\begin{equation}\label{eqn:C1_star}
  C_1^*(n, d_K, \mathfrak{p}) := (n+1) k_2 k_3 k_4.
\end{equation}

\begin{table}[!ht]
  \centering
  \caption{The constants $a^{(1)}$ and $\kappa_1$}
  \label{table:a1_kappa1}
  \begin{tabular}{lccr}
    \toprule
    Case && $a^{(1)}$ & $\kappa_1$ \\[1pt]
    \midrule
    $p = 2$ && $32$ & $40$ \\[2pt]
    $p = 3$ && $16$ & $20$ \\[2pt]
    $p > 3$ and $e_\mathfrak{p} \geq 2$ && $16$ & $20$ \\[2pt]
    $p > 3$ and $e_\mathfrak{p} = 1$ && $\tfrac{8(p-1)}{p-2}$ & $10$ \\
    \bottomrule
  \end{tabular}
\end{table}

\begin{table}[!ht]
  \centering
  \caption{The constant $c^{(1)}$}
  \label{table:c1}
  \begin{tabular}{lcrclcr}
    \toprule
    \multicolumn{3}{c}{$p \leq 5$} && \multicolumn{3}{c}{$p > 5$} \\
    \cmidrule{1-3} \cmidrule{5-7}
    Case && $c^{(1)}$ && Case && $c^{(1)}$ \\
    \midrule
    $p = 2$ && $160$ && $p \equiv 1\ (4)$ and $e_\mathfrak{p} = 1$ && $1473$ \\
    $p = 3$ and $d_K = 1$ && $537$ && $p \equiv 1\ (4)$ and $e_\mathfrak{p} \geq 2$ && $1502$ \\
    $p = 3$ and $d_K \geq 2$  && $759$ && $p \equiv 3\ (4)$, $e_\mathfrak{p} = 1$, $d_K = 1$ && $1288$ \\
    $p = 5$ and $e_\mathfrak{p} = 1$ && $1473$ && $p \equiv 3\ (4)$, $e_\mathfrak{p} = 1$, $d_K \geq 2$ && $1282$ \\
    $p = 5$ and $e_\mathfrak{p} \geq 2$ && $319$ && $p \equiv 3\ (4)$, $e_\mathfrak{p} \geq 2$ && $2190$ \\
    \bottomrule
  \end{tabular}
\end{table}

\subsection{A Remark about implementation}

For this subsection only, suppose all hypotheses in Theorem \ref{thm:Yu_2007_bound} are satisfied, except $K$ does \textbf{not} satisfy Yu's auxiliary condition. Set
\[ K' := \begin{cases} K(\zeta_4) & q = 2, \\ K(\zeta_3) & q = 3. \end{cases} \]
Let $\mathfrak{P}$ be a prime of $\mathscr{O}_{K'}$ above $\mathfrak{p}$. Let $e_{\mathfrak{P} \mid \mathfrak{p}}$ be the ramification index of $\mathfrak{P}$ over $\mathfrak{p}$. Because
\[ \ord_\mathfrak{p} \alpha = e_{\mathfrak{P} \mid \mathfrak{p}} \ord_\mathfrak{P} \alpha, \qquad \alpha \in K^\times, \]
we see $\ord_\mathfrak{P} \mu_j = 0$ for all $j$. Now Theorem \ref{thm:Yu_2007_bound} applies with $K'$ and $\mathfrak{P}$ in place of $K$ and $\mathfrak{p}$, respectively.
\begin{corollary}\label{corollary:K_not_aux}
Under the conditions of this subsection,
\[ \ord_\mathfrak{p} \left( \Theta - 1 \right) < e_{\mathfrak{P} \mid \mathfrak{p}} \cdot C_1^*(n, d_{K'}, \mathfrak{P}) \cdot \Omega'(n, d_{K'}, \mathfrak{P}) \cdot \log B. \]
\end{corollary}
Note that even if $\mathfrak{p}$ splits as $\mathfrak{P}\mathfrak{P}'$ in $K'$, the choice of $\mathfrak{P}$ is irrelevant; both give the exact same bound in the Corollary.

\subsection{Bound of Baker-W\"ustholz}
We now give an effective version of Baker's theorem. (Notations are as in \S\ref{sec:abs_vals}, \ref{sec:ht_funcs}.)
\begin{theorem}[Baker-W\"ustholz, {\cite[pg.~20]{Baker-Wustholz:1993}}]\label{thm:BW}
  Let $L$ be a linear form in $t+1$ indeterminates,
  \[ L(z_0, \dots, z_t) = b_0 z_0 + \cdots + b_t z_t, \qquad b_i \in \Z. \]
Let $B=\max \{ |b_0|, \dots, |b_t| \}$, and let $\rho_0, \dots, \rho_t \in \overline{\mathbb{Q}} - \{0, 1\}$. Let $K^{\prime}$ be the subfield of $\overline{\mathbb{Q}}$ generated by the $\rho_i$. If $B>3$ and
\[ \Lambda = L(\log \rho_0, \log \rho_1, \dots, \log \rho_t) \neq 0, \]
then
\[ \log |\Lambda| > -C(t, d_{K^{\prime}}) \log (B) \prod_{j=0}^t h'(\rho_j), \]
where the constant $C(t, d_{K^{\prime}})$ is defined by
\[ C(t, d_{K^{\prime}}) = 18 (t+2)! (t+1)^{(t+2)} (32 d_{K^{\prime}})^{(t+3)} \log \left( 2(t+1) d_{K^{\prime}} \right). \]
\end{theorem}
Note that we may be sure $\Lambda \neq 0$ if the set $\{ \log \rho_i \}$ is linearly independent over $\mathbb{Q}$.

\subsection{Obtaining the initial bound}
The theorems of Baker-W\"ustholz and Yu both provide inequalities of the form ``a polynomial function of $B$ is bounded by a polynomial function of $\log(B)$,'' which in turn guarantee an absolute bound on $B$. The analysis to determine such a bound explicitly is standard; we will use the following result of Peth\H{o} and de Weger for this purpose.
\begin{lemma}[Peth\H{o} and de Weger {\cite[Lemma 2.2]{Petho-deWeger:1986}}] \label{thm:PdW}
Suppose the real numbers $a, b, h$ satisfy $a\geq 0$, $h\geq 1$, $b>\bigl( \frac{e^2}{h} \bigr)^h$, and let $x\in\mathbb{R}$ be the largest solution to the equation
\[ x = a + b(\log x)^h. \]
Then
\[ x < 2^h \left( a^{\frac{1}{h}} + b^{\frac{1}{h}} \log\left( h^h b \right) \right)^h. \]
\end{lemma}

\section{Initial Exponent Bounds}\label{sec:initial_bound}

\subsection{An upper bound at the extremal place}
Suppose $(\tau_1, \tau_2)$ is a solution to the $S$-unit equation, with $\tau_i$ specified as in \eqref{eqn:tau_i}. We set $B = \max_{i,j} \left| b_{i,j} \right|$, and assume $B \geq \max \{4, w \}$. Relabeling $\tau_1$ and $\tau_2$ if necessary, we assume $B = |b_{1,j}|$ for some $1 \leq j \leq t$. Recall that $S$ contains precisely $t + 1$ places, $\mathfrak{p}_1, \dots, \mathfrak{p}_{t+1}$. We choose the indices $k, \ell \in \{1, 2, \dots, t+1\}$ so that
\[ \bigl| \log |\tau_1|_{\mathfrak{p}_k} \bigr| = \max_{\mathfrak{p} \in S} \, \bigl| \log| \tau_1|_{\mathfrak{p}} \bigr|, \qquad |\tau_1|_{\mathfrak{p}_\ell} = \min_{\mathfrak{p}\in S} \, |\tau_1|_{\mathfrak{p}}.\]

\begin{remark}
In the sequel, we number our constants in an effort to stay consistent with the enumeration given in Smart's paper \cite{Smart:1995}. There, Smart considers a more general unit equation, and so introduces certain constants $c_4(i)$, $c_6(i)$, $c_7(i), \dots$ whose values are trivial in the present application. So while the alert reader may notice gaps in the enumeration of constants, this is intentional. (Adjusting our implementation to the more general setting is not difficult, but we are satisfied to limit the discussion to match the current state of the implementation.)
\end{remark}

For any choice of $U := \{\mathfrak{u}_1,\dots, \mathfrak{u}_t\} \subseteq S$ define the $t\times t$ matrix
\[ M = (m_{i,j}), \qquad m_{i,j}=\log| \rho_j|_{\mathfrak{u}_i}. \]
One may always choose $U$ so that $M$ is invertible (see \cite[\S5.1]{Evertse-Gyory:2015}), and so we assume this is the case. We have
\[ \left( \begin{array}{c} b_{1,1} \\ b_{1,2} \\ \vdots \\ b_{1,t} \end{array} \right) = M^{-1} \left( \begin{array}{c} \log |\tau_1|_{\mathfrak{u}_1} \\ \log|\tau_1|_{\mathfrak{u}_2} \\ \vdots \\ \log|\tau_1|_{\mathfrak{u}_t} \end{array} \right).\]
Let $\| M \|$ be the row norm of $M^{-1}$, i.e. $\| M \| = \max_i \sum_{j=1}^t |m_{i,j}|$, and set
\[ c_1 := \max \left\{ 1,\ \max_{U \subseteq S} \,\{ ||M|| : M \text{ is invertible } \} \right\}. \]
Note that this differs slightly from Smart's definition, to ensure that $c_1\geq 1$. Then $B \leq c_1 \bigl| \log|\tau_1|_{\mathfrak{p}_k} \bigr|$. We define
\[ c_2 := \frac{1}{c_1} \qquad \qquad c_3 := \frac{0.9999999c_2}{r+s}. \]
By \cite[Lemma 2]{Smart:1995}, we have
\begin{equation}\label{eqn:tau1_upperbound}
|\tau_1|_{\mathfrak{p}_\ell}\leq e^{-c_3 B}.
\end{equation}

We now have an upper bound on $|\tau_1|_{\mathfrak{p}_\ell}$ in terms of $B$. We next establish a lower bound, also involving $B$, which will force a limit on the size of $B$. The precise argument depends on whether $\mathfrak{p}_\ell$ is a finite or infinite place. For the purposes of the algorithm, we must compute this bound on $B$ for each possible index $1 \leq \ell \leq t$; we have no choice but to take the largest possible bound, i.e., the larger of the two values $K_0$ and $K_1$ determined in the remainder of this section.

\subsection{Case I: $\mathfrak{p}_\ell$ is finite}\label{sec:pl_fin}
If $\mathfrak{p}_\ell$ is finite, then let $\mathfrak{p}_\ell$ also denote the associated prime ideal in $\OK$. Let $p$ be the prime of $\Z$ lying below $\mathfrak{p}_\ell$, and let $e_\ell$ and $f_\ell$ denote the ramification index and inertial degree of $\mathfrak{p}_\ell$ over $p$, respectively. From \eqref{eqn:tau1_upperbound} we have
\begin{equation}\label{eqn:Np_bound}
N_{K/\Q}( \mathfrak{p}_{\ell} )^{-\ord_{\mathfrak{p}_{\ell}} (\tau_1)} \leq e^{-c_3 B}.
\end{equation}
Setting
\[ c_5(\ell) := \frac{c_3}{e_\ell \log N_{K/\Q}(\mathfrak{p}_\ell)}, \]
the inequality \eqref{eqn:Np_bound} yields
\begin{equation}\label{eq:fin_lb}
\ord_{\mathfrak{p}_{\ell}} \tau_1 \geq \frac{c_3 B}{\log N_{K/\Q}(\mathfrak{p}_{\ell})} = e_\ell c_5(\ell) B > 0,
\end{equation}
and so $\ord_{\mathfrak{p}_\ell} \tau_2 = 0$. We would like to apply Yu's Theorem to ${\Theta=\tau_2}$, but unfortunately the generators $\rho_i$ may have nonzero order with respect to $\mathfrak{p}_\ell$. So we now replace the $\rho_i$ with a different set of generators, as in \cite[pgs.~824--825]{Smart:1995}.
First, set $n_i :=  \ord_{\mathfrak{p}_\ell} \rho_i$. Necessarily, there exist indices $i$ for which $n_i \neq 0$. Choose $i_0$ so that
\[ |n_{i_0}| = \min \{ |n_i| : n_i \neq 0 \}, \]
and now relabel so that $i_0 = t$. For $1 \leq i \leq t-1$, define
\[ \mu_i = \rho_i^{n_t} \rho_t^{-n_i}, \]
so that $\ord_{\mathfrak{p}_\ell} \mu_i = 0$. Next, for each $i$ with $1 \leq i \leq t-1$, choose integers $d_i, r_i$ such that
\[ 0 \leq r_i < |n_t| \qquad \text{and} \qquad b_{2,i} = n_t d_i + r_i \]
Necessarily, $|d_i| \leq B$. Set $N := \sum_{i=1}^{t-1} n_i r_i$.
\begin{lemma}
  We have $N \equiv 0 \pmod{n_t}$.
\end{lemma}
\begin{proof}
Since $\ord_{\mathfrak{p}_\ell} (\tau_2 \rho_0^{-b_{2,0}}) = 0$, we know $\sum_{i=1}^t n_i b_{2,i} = 0$. Thus,
\[ n_t b_{2,t} = - \sum_{i=1}^{t-1} n_i b_{2,i} = - n_t \sum_{i=1}^{t-1} n_i d_i  - \sum_{i=1}^{t-1} n_i r_i, \]
proving the claim.
\end{proof}
Setting $N_0 = \frac{N}{n_t}$ and $\mu_0 := \rho_0^{b_{2,0}} \rho_t^{-N_0} \cdot \prod_{i=1}^{t-1} \rho_i^{r_i}$, we have arranged that
\begin{equation}\label{eq:tau2_in_mus}
  \tau_2 = \mu_0 \prod_{i=1}^{t-1} \mu_i^{d_i}, \quad |d_i| \leq B, \quad \ord_{\mathfrak{p}_\ell} \mu_i = 0.
\end{equation}
Since $0 \leq b_{2,0} < w$ and $0 \leq r_i < |n_t|$, there are only finitely many possible values for $\mu_0$, and this finite set can be determined without any knowledge of $B$ or the $b_{2,i}$. For each $\mu_0$, we may apply Corollary \ref{corollary:Yu_bound} or \ref{corollary:K_not_aux} as appropriate, and obtain a constant $c'_8(\ell, \mu_0)$ such that
\[ \ord_{\mathfrak{p}_\ell} \tau_1 = \ord_{\mathfrak{p}_\ell} (\tau_2 - 1) < c'_8(\ell, \mu_0) \log B. \]
Setting
\[ c_8(\ell) := \max \left\{ \frac{e^2}{\log 2}, \max_{\mu_0} \{ c'_8(\ell, \mu_0) \} \right\}, \]
we may be sure every $S$-unit solution satisfies
\begin{equation}\label{eq:fin_ub}
\ord_{\mathfrak{p}_\ell} \tau_1 < c_8(\ell) \log B.
\end{equation}
Combining inequalities \eqref{eq:fin_lb} and \eqref{eq:fin_ub}, we have
\[ B < \frac{c_8(\ell)}{e_\ell c_5(\ell)} \log B. \]
Since $c_1 \geq 1$ and $c_8(\ell) \geq e^2 (\log 2)^{-1}$, it follows that
\[ \frac{c_8(\ell)}{e_\ell c_5} \geq e^2. \]
Applying Lemma \ref{thm:PdW} with $a = 0$, $b = c_8(\ell)/e_\ell c_5(\ell)$, and $h = 1$, we may conclude
\[ B \leq K_0(\ell) := \frac{ 2 c_8(\ell) }{e_\ell c_5(\ell)} \log \left( \frac{ c_8(\ell) }{e_\ell c_5} \right). \]
Set
\[ K_0 := \max \{ K_0(\ell) : \mathfrak{p}_\ell \in S_\mathrm{fin} \}. \]
If $\ell$ corresponds to a finite place, then $B \leq K_0$.

In our implementation, the functions {\texttt{mus}} and {{\texttt{possible\_mu0s}} are used to recover the $\mu_i$ for each finite place $\mathfrak{p}_\ell$. The constants $c_8(\ell)$ determined from Yu's Theorem are computed in \texttt{Yu\_bound}, while the constant $K_0$, which may be of independent interest, is computed by {\texttt{K0\_func}}.

\subsection{Case II: $\mathfrak{p}_\ell$ is infinite}\label{sec:pl_inf}
We now assume $\mathfrak{p}_\ell$ is infinite. As in \S2.3, we let $\sigma_{\mathfrak{p}_\ell}$ denote the embedding of $K$ into $\C$ such that \[ \left| \alpha \right|_{\mathfrak{p}_\ell} = \left| \sigma_{\mathfrak{p}_\ell}(\alpha) \right|^{\delta(\ell)}, \quad \text{where } \delta(\ell) = \begin{cases} 1 & \text{$\mathfrak{p}_\ell$ is real,} \\ 2 & \text{$\mathfrak{p}_\ell$ is complex.} \end{cases} \]
We let $\alpha^{(\ell)}$ denote $\sigma_{\mathfrak{p}_\ell}(\alpha)$ for any $\alpha \in K$, and we define
\[ c_{11}(\ell) := \frac{ \delta(\ell) \log 4}{c_3}, \qquad c_{13}(\ell) := \frac{ c_3 }{\delta(\ell)}. \]
The condition \eqref{eqn:tau1_upperbound} can now be expressed as
\[ \left| \tau_1^{(\ell)} \right| \leq e^{-c_{13}(\ell) B}. \]
The choices of $c_{11}(\ell)$ and $c_{13}(\ell)$ guarantee that
\[ B \geq c_{11}(\ell) \quad \Longrightarrow \quad \left| \tau_1^{(\ell)} \right| \leq \frac{1}{4}. \]
Set $\Lambda := \log \tau_2^{(\ell)}$. The estimate $|\log z| \leq 2|z-1|$ holds for $|z-1| \leq \frac{1}{4}$, and so
\begin{equation}\label{eqn:log_tau2_bound}
\left| \Lambda \right| \leq 2 \left| \tau_2^{(\ell)} - 1 \right| = 2 \left| \tau_1^{(\ell)} \right| \leq 2e^{-c_{13}(\ell)B}.
\end{equation}
The next step is to view $\Lambda$ as a linear form in logarithms and apply the theorem of Baker and W\"{u}stholz. Set $\zeta := \exp \frac{2 \pi \sqrt{-1}}{w} \in \C$. Since $\rho_0$ is a $w$th root of unity, there exists $0 \leq k < w$ such that $(\rho_0^{(\ell)})^{b_{2,0}} = \zeta^k$. By \eqref{eqn:tau_i}, we have
\begin{equation}\label{eq:Lambda_linear_logs}
\begin{split}
\Lambda & = \log \left( \left( \rho_0^{(\ell)} \right)^{b_{2,0}} \cdot \prod_{j=1}^t \left(\rho_j^{(\ell)}\right)^{b_{2,j}} \right) \\
& = \log \zeta^k + \sum_{j=1}^t b_{2,j} \log \rho_j^{(\ell)} + A \cdot 2\pi\sqrt{-1} \\
& = k \log \zeta + \sum_{j=1}^t b_{2,j} \log \rho_j^{(\ell)} + Aw \log \zeta \\
& = (Aw + k) \log \zeta + \sum_{j=1}^t b_{2,j} \log \rho_j^{(\ell)},
\end{split}
\end{equation}
where we have introduced $A \in \Z$ to adjust for the principal branch of the logarithm. Certainly $|A| \leq tB$, and so $|Aw + k| \leq (t + 1)Bw$. Set
\[ b'_{2,j} := \begin{cases} Aw + k & j = 0 \\ b_{2,j} & j > 0 \end{cases} \]
and $L'(z_0,\dots,z_t) := \sum_{j=0}^t b'_{2,j} z_j$. We now have
\[ |\Lambda| = \left| L'(\log \zeta, \log \rho_1^{(\ell)}, \dots, \log \rho_t^{(\ell)}) \right|. \]
Taking $K' = \Q(\rho_0, \dots, \rho_t) \cong \Q(\zeta, \rho_1^{(\ell)}, \dots, \rho_t^{(\ell)})$, we define
\[ c_{14}(\ell) := C(t, d_{K'}) \prod_{j=0}^t h'(\rho_j). \]
(Recall that $C(t, d_{K'})$ is defined in Theorem \ref{thm:BW}.)  We have $|b'_{2,j}| \leq B' := (t + 1)Bw$. Applying Theorem \ref{thm:BW} to $\Lambda$, we obtain
\[ \log |\Lambda| > -c_{14}(\ell) \cdot \log B' = -c_{14}(\ell) \log \bigl( (t+1)wB \bigr). \]
Combining this inequality with \eqref{eqn:log_tau2_bound}, we obtain
\begin{equation}\label{eqn:lambda}
 2e^{-c_{13}(\ell)B} \geq |\Lambda| \geq e^{-c_{14}(\ell) \log B'}.
\end{equation}
This yields the inequality
\[ B < a(\ell) + b(\ell) \log B, \]
where
\[ a(\ell) := \frac{1}{c_{13}(\ell)} \left( \log 2 + c_{14}(\ell) \log \bigl( (t+1)w \bigr) \right), \quad b(\ell) := \frac{ c_{14}(\ell) }{ c_{13}(\ell) }. \]
As $c_{13}(\ell) \leq \frac{1}{t}$ and $c_{14}(\ell) \geq 32^3$, we have $a(\ell) \geq 0$ and ${b(\ell) \geq e^2}$. So by Lemma \ref{thm:PdW}, $B < c_{15}(\ell)$ (provided $B \geq c_{11}(\ell)$), where
\[ c_{15}(\ell) := 2\bigl( a(\ell) + b(\ell) \log b(\ell) \bigr). \]
Thus, setting
\[ \begin{split}
K_1(\ell) & := \max \{ c_{11}(\ell), c_{15}(\ell) \}, \\
K_1 & := \max \{K_1(\ell) : \mathfrak{p}_\ell \text{ is infinite} \},
\end{split} \]
we may be sure $B \leq K_1$. In our implementation, the constant $K_1$ is computed in the function {\texttt{K1\_func}}.

Combining all the results of this section, we obtain the following.
\begin{lemma}\label{lemma:bound_for_B}
  The constant $B$ satisfies $B \leq \max \{4, w, K_0, K_1\}$.
\end{lemma}}

\section{LLL Reduction}\label{sec:LLL_red}

In this section we explain how we can reduce the upper bound we have computed in Section \ref{sec:initial_bound}. This is necessary, because in practice the size of the initial bound is extremely large and cannot be used for practical computations. The idea of the method we will present here has its origin in de Weger's thesis \cite{deWeger:thesis, deWeger:1987, deWeger:1989} where he develops a method based on multi-dimensional approximation lattices of linear form of $p$-adic numbers to solve (among many other equations) $S$-unit equations\footnote{It is worth mentioning the recent results of von K\"{a}nel and Matschke \cite{vonKanel-Matschke:2016}, who solve $S$-unit equations using modularity.} over $\Q$. These ideas of de Weger have been extended by himself and others to apply over any number field $K$, and have also been used for the solution of other exponential Diophantine equations \cite{Tzanakis-deWeger:1989, Tzanakis-deWeger:1991, Tzanakis-deWeger:1992, Smart:1995}.

In the reduction step we use the LLL reduction algorithm on lattices generated by integer matrices. So instead of the classical LLL algorithm \cite{LLL:1982}, we use the algorithm in \cite{deWeger:1987}. If $\Lat$ is a lattice in $\R^n$, let $\Lat^* = \Lat - \{\mathbf{0} \}$. For  $\by\in\R^n$, we define
\[ \ell( \Lat, \by) = \begin{cases}
    \displaystyle \min_{\bx \in \Lat^*} \Vert \bx \Vert, & \text{if}~\by \in \Lat,  \\
    \displaystyle \min_{\bx \in \Lat} \Vert \bx - \by \Vert,     & \text{otherwise}.
  \end{cases}
\]
Computing the exact value of $\ell(\Lat,\by)$ is a very challenging problem in general. Instead, the function \texttt{minimal\_vector} computes a lower bound using standard properties of a reduced basis of a lattice and the LLL algorithm (see \cite[Chapter V]{Smart:1998}). As in the previous section, we follow Smart's notation in \cite{Smart:1995}. Most of the material we present in this section can also be found in \cite{Smart:1998, Evertse-Gyory:2015}.

We preserve the meaning of $\mathfrak{p}_\ell$ from \S4. When $\mathfrak{p}_\ell$ is a finite place, we let $p$ denote the prime of $\Z$ lying below $\mathfrak{p}_\ell$. We continue to assume $B \geq \max \{4, w\}$ in this section.

\subsection{Finite places}
Suppose $\mathfrak{p}_\ell$ is a finite place. Set
\[ c_{16}(\ell) := 1 + \frac{1}{c_5(\ell)}, \]
and suppose that $B \geq c_{16}(\ell)$. Define $\Delta_2 \in K_{\mathfrak{p}_\ell}$ as $\Delta_2 := \log_p \tau_2$.
Combining \eqref{eq:fin_lb}, \eqref{eq:ord_p_ord_frak_p}, and ${B \geq c_{16}(\ell)}$, shows that $\ord_p \tau_1 > 1$. Consequently, $|\tau_1|_p < p^{-\frac{1}{p-1}}$, and by \eqref{eq:ord_logp},
\[ \ord_p \Delta_2 = \ord_p \log_p \tau_2 = \ord_p \log_p (1 - \tau_1) = \ord_p \tau_1 > 1. \]
Let $\mu_i, d_i$ be as given in \eqref{eq:tau2_in_mus}, so that we have
\[ \Delta_2 = \log_p \tau_2 = \log_p \mu_0 + \sum_{i=1}^{t-1} d_i \log_p \mu_i. \]

Choose $\theta \in K_{\mathfrak{p}_\ell}$ such that $K_{\mathfrak{p}_\ell} = \Q_p(\theta)$, and let $\Disc(\theta)$ denote the discriminant of $\theta$. Set $D_p(\theta) = \ord_p \Disc (\theta)$ and $n = [K_{\mathfrak{p}_\ell} : \Q_p]$, so that $n = e_\ell f_\ell$. Expressing $\Delta_2$ with respect to the power basis, we obtain $\Delta_{2,k} \in \Q_p$ such that $ \Delta_2 = \sum_{k=0}^{n-1} \Delta_{2,k} \theta^k$. Further, we may express
\begin{equation}\label{eqn:Delta_2i}
\Delta_{2,k} = a_{0,k} + \sum_{j=1}^{t-1} d_j a_{j,k}, \qquad a_{j,k} \in \Q_p,~0 \leq k \leq n-1.
\end{equation}
Using an idea due to Evertse \cite[p.257]{Tzanakis-deWeger:1992}, we have
\[ \ord_p \Delta_{2,k} \geq c_5(\ell) B - \frac{D_p(\theta)}{2}. \]
Define
\[ \begin{split}
c_{17}(\ell) & := \min \,\{ \ord_p a_{j,k} : 1 \leq j \leq t-1,~0 \leq k \leq n-1 \}, \\
c_{18}(\ell) & := c_{17}(\ell) + \frac{D_p(\theta)}{2}, \end{split} \]
and choose $\lambda \in \Q_p$ such that $\ord_p \lambda = c_{17}(\ell)$.

Should there be some index $k$ such that $c_{17}(\ell) > \ord_p(a_{0,k})$, then $\ord_p \Delta_{2,k} = \ord_p a_{0,k} < c_{17}(\ell)$, and consequently
\[ B < \frac{c_{18}(\ell)}{c_5(\ell)}. \]
For the remainder, then, we assume
\[ c_{17}(\ell) \leq \min \,\{ \ord_p a_{0,k} : 0 \leq k \leq n-1 \}. \]
By the choice of $\lambda$, $\kappa_{j,k} := a_{j,k}/\lambda$ is a $p$-adic integer for all $j,k$, and we may rewrite \eqref{eqn:Delta_2i} as
\[ \frac{ \Delta_{2,k} }{ \lambda } = \kappa_{0,k} + \sum_{j=1}^{t-1} d_j \kappa_{j,k}, \quad \text{with} \quad \ord_p \left( \frac{ \Delta_{2,k} }{\lambda} \right) \geq c_5(\ell)B - c_{18}(\ell). \]

For any $a \in \Z_p$ and a positive integer $z$, let $a^{(z)}$ denote the unique integer between $0$ and $p^z$ such that $a\equiv a^{(z)}\pmod{p^z}$. For a positive integer $u$, let $\Lat$ be the lattice generated by the columns of the matrix
\[
  \begin{pmatrix}
    1                    &        &  0                     & 0    &  \cdots      & 0    \\
                         & \ddots &                        & \vdots    &        &   \vdots  \\
    0                    &        & 1                      & 0    &   \cdots     & 0    \\
    \kappa^{(u)}_{1, 0}   & \cdots & \kappa^{(u)}_{t-1, 0}   & p^u &        & 0   \\
    \vdots               &        & \vdots                 &     & \ddots &     \\
    \kappa^{(u)}_{1, n-1} & \cdots & \kappa^{(u)}_{t-1, n-1} & 0   &        & p^u
  \end{pmatrix}\in\Z^{(t+n-1)\times (t+n-1)}.
\]
Define
\[
  \by = \begin{pmatrix}
    0 & \cdots & 0 & -\kappa^{(u)}_{0, 0} & \cdots & -\kappa^{(u)}_{0, n-1} \\
  \end{pmatrix}^\intercal \in \Z^{t+n-1}.
\]
Also set
\[ K_0^{\mathrm{LLL}}(\ell) := \max \left\{ 4, w, \frac{u + c_{18}(\ell)}{c_5(\ell)}, c_{16}(\ell) \right\}. \]
The following lemma is a restatement of \cite[Lemma 5]{Smart:1995}, and provides an opportunity to improve the bound on $B$.\footnote{Note $s_d$ in \cite{Smart:1995} has the value $t-1$ in our notation.}

\begin{lemma}\label{lem:p_adic_reduction}
If $\ell(\Lat,\by) > \sqrt{t-1} \cdot K_0$, then $\displaystyle B < K_0^\mathrm{LLL}(\ell)$.
\end{lemma}
In the function \texttt{p\_adic\_LLL\_bound} we have implemented the above analysis. In more detail, the functions \texttt{log\_p} and \texttt{embedding\_to\_Kp} are used to compute the constants $a_{j,k} \in \Q_p$ up to a given precision. If this precision is $M$, i.e., if the $a_{j,k}$ are stored as pairs of integers modulo $p^M$, we clearly require $M > u$ for the algorithm to be meaningful. However, the shift by $\lambda$ requires an additional $c_{17}(\ell)$ $p$-adic digits of precision. So the algorithm checks that $M > u + c_{17}(k)$; if this fails, then the $p$-adic logarithms are computed to higher precision and the process is repeated.

The function \texttt{log\_p} is based on an algorithm of Smart in \cite[p.~30]{Smart:1998}. However, our implementation also resolves a crucial computational problem in the evaluation of $\log_p$ that to our knowledge has not been mentioned in the literature. To understand the issue, we must describe carefully what ``computing the logarithm'' means in a $p$-adic setting. Let us view $K$ as a subfield of $K_\mathfrak{p}$, and specify $\omega_1, \dots, \omega_n \in K$, a $\Q_p$-basis for $K_\mathfrak{p}$. Suppose $\alpha \in K$ and set $\beta := \log_p \alpha \in K_\mathfrak{p}$. Necessarily, there exist $b_j \in \Q_p$ such that $\beta = \sum_j b_j \omega_j$.

As a practical matter, no algorithm can return the true value $\beta$; it can only  return $\tilde{\beta} \in K$, an approximation to $\beta$ with $|\beta - \tilde{\beta}|_\mathfrak{p}$ very small. In practice however, we require something more specific. We want to find $\tilde{b}_j \in \Q$ such that $|b_j - \tilde{b}_j|_p$ is small for each $j$. When $p$ splits in $K$, the original algorithm is not guaranteed to do this.

It can happen that at some other prime $\mathfrak{p}'$ above $p$, $\alpha$ has a negative valuation. Consequently, the sum \eqref{eqn:log_series} used to compute $\tilde{\beta}$ will \emph{not} converge $\mathfrak{p}'$-adically, and the approximations $\tilde{b}_j$ are not guaranteed to be $p$-adically close to the $b_j$. We resolve this problem by choosing a suitable element $\eta \in K$ with $\ord_{\mathfrak{p}'} \eta \geq 0$ for all $\mathfrak{p'} \mid p$, such that $\eta$ also satisfies
\[ \ord_{\mathfrak{p}'} (\eta \alpha) \geq 0, \qquad \ord_{\mathfrak{p}} (\eta \alpha) = \ord_{\fp} \eta  = 0. \]
Then it holds that
\[ \log_p \alpha = \log_p( \eta \alpha) - \log_p \eta. \]
By evaluating the difference on the right hand side, these $\mathfrak{p}'$-adic divergence issues are avoided, and we may be sure that the individual coefficients $\tilde{b}_j$ approximate the $b_j$ $p$-adically.

In \texttt{p\_adic\_LLL\_bound\_one\_prime}, we attempt to find a value $u$ such that Lemma \ref{lem:p_adic_reduction} applies. If successful, we record the improved bound
$K_0^\mathrm{LLL}(\ell)$. The improvement offered by Lemma \ref{lem:p_adic_reduction} depends only on the assumption that $\mathfrak{p}_\ell$ is the extremal place, and that some bound $K_0 \geq c_{16}(\ell)$ on the exponents is known. So we may replace $K_0$ by $K_0^\mathrm{LLL}(\ell)$ and attempt to apply Lemma \ref{lem:p_adic_reduction} again, possibly improving the bound further. Because the application of LLL is very fast compared to the sieving step described in \S\ref{sec:sieve}, the algorithm repeats this process until no further improvements can be made to $K_0^\mathrm{LLL}(\ell)$. Once each $K_0^\mathrm{LLL}(\ell)$ has been optimized in this way, the function \texttt{p\_adic\_LLL\_bound} returns
\[ K_0^\mathrm{LLL} := \max \{K_0^\mathrm{LLL}(\ell) : \mathfrak{p}_\ell \in S_\mathrm{fin} \}. \]

\subsection{Complex places}\label{sec:LLL_complex_case}
We now consider the case where $\mathfrak{p}_\ell$ is an infinite complex place. The reduction is quite analogous to the $p$-adic case; again the standard references are \cite{Smart:1995, Smart:1998, Evertse-Gyory:2015}. We keep the notations from \S\ref{sec:pl_inf}. For $0 \leq j \leq t$ we define the complex numbers
\[ \kappa_j := \begin{cases} \log \zeta & j = 0, \\ \log \rho_j^{(\ell)} & j > 0. \end{cases} \]
As $\mathfrak{p}_\ell$ is an infinite place, we have established already that the $b'_{2,j}$ in
\[ \Lambda = \log \tau_2^{(\ell)} = \sum_{j=0}^t b'_{2,j} \kappa_j \]
satisfy the bounds
\[ \qquad |b'_{2,0}| \leq (t+1)Bw, \qquad |b'_{2,j}| \leq B\ \text{for}\ 1 \leq j \leq t. \]
We now attempt to use lattice reduction to improve the bound; the choice of lattice and certain constants will depend slightly on whether the $\kappa_j$ are all purely imaginary. So we define
\begin{equation*}
  \sigma := \begin{cases} 1 & \text{every $\kappa_j$ is pure imaginary,} \\
  0 & \text{otherwise}, \end{cases}
\end{equation*}
and define
\[ S := \bigl( t - 1 + \sigma \bigr) K_1^2, \qquad T := \left( \frac{1}{\sqrt{2}} \right)^{1 + \sigma} (t + w + tw)K_1. \]
If $\kappa_1,\dots,\kappa_t$ are not all pure imaginary, relabel $\kappa_1,\dots,\kappa_t$ so that $\Re \kappa_t \neq 0$. Now define
\[ a_{tt} := \begin{cases} [C \Re \kappa_t] & \Re \kappa_t \neq 0, \\ 1 & \Re \kappa_t = 0. \end{cases} \]
Let $A$ be the $(t + 1) \times (t + 1)$ integer matrix
\begin{equation}\label{eq:lattice_matrix}
A =   \begin{pmatrix}
    1                     &        & 0      &   0                    &         0     \\
                          & \ddots &        &   \vdots               &       \vdots  \\
    0                     &        & 1      &   0                    &         0     \\
    [C \cdot \Re \kappa_1] & \cdots & [C \cdot \Re \kappa_{t-1}] & a_{tt} & 0             \\
    [C \cdot \Im \kappa_1] & \cdots & [C \cdot \Im \kappa_{t-1}] & [C \cdot \Im \kappa_t] & [C \cdot \frac{2\pi}{w}]
  \end{pmatrix}.
\end{equation}
(By design, the upper left $t \times t$ block of $A$ is the identity matrix in case the $\kappa_j$ are all pure imaginary.) Now, let $\Lat$ be the lattice generated by the columns of $A$, and suppose $m_\Lat$ is a positive lower bound for $\ell(\Lat, \mathbf{0})$.
When $\mathfrak{p}_\ell$ is a infinite non-real place, we define:
\[ K_1^\mathrm{LLL}(\ell) := \max \left\{4, w, \frac{1}{c_{13}(\ell)} \log \left( \frac{2C}{ (m_\Lat^2 - S)^\frac{1}{2} - T} \right) \right\}. \]
Similar to \cite[Lemma VI.2]{Smart:1998}, we have
\begin{lemma}\label{lem:complex_reduction}
Suppose $\mathfrak{p}_{\ell}$ is a non-real infinite place. With notation as above, suppose $C$ is chosen such that $m_\Lat^2 > T^2 + S$. Then $B \leq K_1^{\mathrm{LLL}}(\ell)$.
\end{lemma}
\begin{proof}
There are two cases to consider, as $\sigma = 0$ or $\sigma = 1$. In each case, our goal is to establish the inequality
\begin{equation}\label{eq:root_mST}
  \sqrt{m_\Lat^2 - S} - T \leq 2C e^{-c_{13}(\ell)B},
\end{equation}
for the result follows by isolating $B$ in the inequality \eqref{eq:root_mST}.

If the $\kappa_j$ are not all pure imaginary, we define
\begin{align*}
 \Phi_1 & := \sum_{j=1}^t b'_{2,j}  \left[ C\cdot\Re(\kappa_j) \right], \\
 \Phi_2 & := b'_{2,0} \left[ C \cdot \tfrac{2\pi}{w} \right] + \sum_{j=1}^t b'_{2,j}  \left[ C\cdot\Im(\kappa_j) \right].
\end{align*}
Then note that
\[ |C\Lambda-(\Phi_1+\Phi_2\sqrt{-1})|\leq T. \]
Therefore
\[ |(\Phi_1+\Phi_2\sqrt{-1})|\leq T+|C\Lambda|.\]
We know from \eqref{eqn:lambda} that $ 2e^{-c_{13}(\ell)B} \geq |\Lambda| $, so
\[ |(\Phi_1+\Phi_2\sqrt{-1})|\leq T+2Ce^{-c_{13}(\ell)B} .\]
Now notice that the vector \[\by=(b'_{2,1}, b'_{2,2},\dots, b'_{2,t-1},\Phi_1, \Phi_2)^\intercal\] is in the lattice $\Lat$, so $|\by|\geq m_{\Lat}$.  Further,
\begin{eqnarray*}
m_{\Lat}^2&\leq & |\by|^2=\sum_{j=1}^{t-1}(b'_{2,j})^2+\Phi_1^2+\Phi_2^2 \\
&\leq &(t-1)K_1^2 + \left| \Phi_1+\Phi_2\sqrt{-1} \right|^2 \\
&\leq &  S +  (T+2Ce^{-c_{13}(\ell)B})^2,
\end{eqnarray*}
which implies \eqref{eq:root_mST} and the result follows.

In case the $\kappa_j$ are all pure imaginary, the approach is similar. Set
\[\Phi=b'_{2,0}[C\cdot \frac{2\pi}{w}]+\sum_{j=1}^t b'_{2,j}  [C\cdot\Im(\kappa_j)]. \]
Similar to the other case, we have $\left| C \Lambda- \left( \Phi\sqrt{-1} \right) \right| \leq T$, and therefore $|\Phi| \leq T + |C\Lambda|$. Again applying  \eqref{eqn:lambda} we obtain
\[ |\Phi|\leq T+2Ce^{-c_{13}(\ell)B}. \]
Now notice that the vector \[\by=(b'_{2,1}, b'_{2,2},\dots, b'_{2,t},\Phi)^\intercal\] is in the lattice $\Lat$, so $|\by|\geq m_{\Lat}$.  Further,
\begin{eqnarray*}
m_{\Lat}^2&\leq & |\by|^2=\sum_{i=1}^{t}(b'_{2,i})^2+\Phi^2\\
&\leq &tK_1^2+|\Phi|^2\\
&\leq &  S +  (T+2Ce^{-c_{13}(\ell)B})^2.
\end{eqnarray*}
Again this implies \eqref{eq:root_mST}.
\end{proof}

\subsection{Real places}\label{sec:LLL_real_case}
Now suppose that $\mathfrak{p}_{\ell}$ is a real infinite place. Although the arguments in \S\ref{sec:LLL_complex_case} apply to $\mathfrak{p}_\ell$, we can obtain a stronger improvement by analyzing this case separately. Replacing $\rho_j$ by $-\rho_j$ as necessary, we may assume $\rho_j^{(\ell)} > 0$ for all $j$ with $1 \leq j \leq t$.

The mere existence of a real place forces $w = 2$ and $0 \leq b_{2,0} \leq 1$. We set $\kappa_0 := \pi \sqrt{-1}$ and define the real numbers
\[ \kappa_j := \log \rho_j^{(\ell)}, \qquad 1 \leq j \leq t. \]
As the $\kappa_j \in \R$, we may revisit \eqref{eq:Lambda_linear_logs}; this time we obtain
\[ \Lambda = \log \tau_2^{(\ell)} = \sum_{j=0}^t b_{2,j} \kappa_j, \]
as no adjustments are required to accommodate the branch cut of the logarithm. Set
\[ S := (t-1) K_1^2, \qquad T := \tfrac{1}{2} (tK_1 + 1), \]
and again let $\Lat$ be generated by the columns of the matrix $A$ in \eqref{eq:lattice_matrix}. When $\mathfrak{p}_\ell$ is an infinite real place, we define:
\[ K_1^\mathrm{LLL}(\ell) := \max \left\{4, w, \frac{1}{c_{13}(\ell)} \log \left( \frac{2C}{ (m_\Lat^2 - S)^\frac{1}{2} - T} \right) \right\}. \]
\begin{lemma}\label{lem:real_reduction}
Suppose that $\mathfrak{p}_{\ell}$ is a real infinite place. With notation and definitions as above, suppose $C$ is chosen so that $m_\Lat^2 > T^2 + S$. Then $B \leq K_1^{\mathrm{LLL}}(\ell)$.
\end{lemma}
\begin{proof}
If we define
\[ \Phi_1 := \sum_{j=1}^t b_{2,j} [C \cdot \kappa_j], \quad \Phi_2 := b_{2,0} [C \cdot \pi], \]
then we obtain
\[ |C \Lambda - \Phi_1 - \Phi_2\sqrt{-1}| \leq T. \]
Observing that the vector
\[ \by = (b_{2,1},\ b_{2,2},\ \cdots,\ b_{2,t-1},\ \Phi_1,\ \Phi_2)^\intercal \in \Lat, \]
and that $|b_{2,j}| \leq B$ for $j > 0$, $|b_{2,0}| \leq 1$, the remainder of the proof now follows the logic of Lemma \ref{lem:complex_reduction} exactly.
\end{proof}

\subsection{Implementation}\label{sec:implementation}
The function \texttt{minimal\_vector} is used in the implementation to compute a value for $m_\Lat^2$.  In \texttt{cx\_LLL\_bound}, we have implemented the reduction step for the infinite places applying the above idea. As in the finite case, the parameter $C$ is chosen inside the function and changed as necessary to meet the bound $m_\Lat^2 > T^2 + S$ (keeping in mind, of course, that the definitions of $S$ and $T$ depend on the particular place $\mathfrak{p}_\ell$). Notice that the proof of Lemmas \ref{lem:complex_reduction} and \ref{lem:real_reduction} depend on obtaining true rounding in obtaining the coefficients of the matrix $A$.  In our implementation, we increase precision until this is assured. Similar to the case where $\mathfrak{p}_\ell$ is finite, the improvement of Lemma \ref{lem:complex_reduction} needs only the assumption that $\mathfrak{p}_\ell$ is the extremal place and that some bound $K_1$ on the exponents is known. So we apply Lemma \ref{lem:complex_reduction} repeatedly until no further improvement to $K_1^\mathrm{LLL}(\ell)$ is possible. Once this has been done for each infinite place, we set
\[ K_1^\mathrm{LLL} := \max \{ K_1^\mathrm{LLL}(\ell) : \mathfrak{p}_\ell \in S_\infty \}. \]
Consequently, we have the following bound which may be passed to the sieve in the next section.
\begin{lemma}\label{lem:LLL_bounds_combined}
  Assume that for each $\ell$, a value $u = u(\ell)$ or $C = C(\ell)$ exists for which the hypotheses of one of the Lemmas \ref{lem:p_adic_reduction}, \ref{lem:complex_reduction}, or \ref{lem:real_reduction} are met. Then the maximum exponent $B$ appearing in any solution $(\tau_1, \tau_2)$ of the $S$-unit equation \eqref{introeqn} satisfies
  \begin{equation}\label{eqn:LLL_ineq}
  B \leq K^\mathrm{LLL} := \max \left\{ K_0^\mathrm{LLL}, K_1^\mathrm{LLL} \right\}.
  \end{equation}
\end{lemma}
We use the proof as an opportunity to summarize the algorithm, up to the sieving step of the next section.
\begin{proof}
There is nothing to show if the solution set is empty, so let us assume otherwise.
We know the $S$-unit equation has only finitely many solutions. Keeping the notation of \eqref{eqn:tau_i}, let $(\tau_1, \tau_2)$ be a solution where $B = |\mathbf{a}_i|$ is maximized. One of the places in $S$, say $\mathfrak{p}_\ell$, is extremal. If $\mathfrak{p}_\ell$ is finite, then the work in \S\ref{sec:pl_fin} demonstrates ${B \leq \max \{4, w, K_0(\ell)\}}$ by applying one of the corollaries deduced from Yu's bound. If $\mathfrak{p}_\ell$ is infinite, then the work in \S\ref{sec:pl_inf} demonstrates ${B \leq \max \{4, w, K_1(\ell)\}}$ by applying the theorem of Baker-W\"ustholz. This establishes an absolute bound on $B$.

For each possible $\ell$, the techniques of this section attempt to replace this absolute bound with a smaller bound. There is no mathematical proof that the lattice reduction techniques will succeed, i.e., that there will exist appropriate values $u$ and $C$ for which Lemmas \ref{lem:p_adic_reduction}, \ref{lem:complex_reduction}, or \ref{lem:real_reduction} apply. However, when they do exist, the improved bound is provably correct by the same lemmas. Here, such success is presumed for every $\ell$, and \eqref{eqn:LLL_ineq} holds.
\end{proof}
In practice, if the hypotheses of Lemma \ref{lem:LLL_bounds_combined} are not established, then one only has the weaker bounds coming from linear forms of logarithms -- these are simply too large to allow for a provably complete search. However, the sieve described in the next section can still be used up to any prescribed bound $B_0$; it will find all solutions satisfying $|\mathbf{a}_i| \leq B_0$.

\section{Further Reducing the Search Space: Sieving}\label{sec:sieve}

The approach taken here, for sieving against primes outside of $S$, is based on an algorithm described by Smart in \cite{Smart:1995}. Smart credits Tzanakis and de Weger with this approach \cite{Tzanakis-deWeger:1991}; Tzanakis reports that these ideas date back to Andrew Bremner.

\subsection{Setup for the sieve}
Recalling the notations of \S\ref{sec:sols_Sunit}, we define for any $m > 0$,
\[ A_{K,S,m} := (\Z/w\Z) \times (\Z/m\Z)^t.\]
This finite set will provide a useful search space for exponent vectors in a way we will make more precise below. There is an obvious surjective map $\pi_m \colon A_{K,S} \to A_{K,S,m}$. Despite the fact that this map is the identity (and not a reduction map) in the $0$th coordinate, we will refer to this as the \emph{reduction modulo $m$} map, and call an element $\mathbf{a} \in A_{K,S,m}$ an \emph{exponent vector modulo $m$}.

Let $\tau \in \OKS^\times$. The \emph{exponent vector} for $\tau$ (relative to $\brho$) is $\Phi_\brho^{-1}(\tau)$. That is, it is the unique $\mathbf{a} \in A_{K,S}$ such that $\tau = \brho^\mathbf{a}$. Given any bound $B$ for the exponent vector of a $\tau \in X_{K,S}$, we obtain a finite subset of $\OKS^\times$ that contains every solution of the $S$-unit equation. Unfortunately, this is usually still too large of a search space to be practical (see \S\ref{sec:observ}), so we must sieve this finite set (or rather, the equivalent finite set of exponent vectors) prior to the exhaustive search. The sieve attempts to provide an efficient solution to the following problem:
\begin{problem}
Find a small set $Y_{K,S}$ satisfying $E_{K,S} \subseteq Y_{K,S} \subseteq A_{K,S}$.
\end{problem}
If we can find a small enough superset $Y_{K,S}$ in a fast enough way, the $S$-unit equation solutions can then be found by brute force search over $Y_{K,S}$.

Suppose $\mathbf{a} \in A_{K,S}$. We call $\mathbf{b} \in A_{K,S}$ a \emph{complement vector} for $\mathbf{a}$ if $\uprho^\mathbf{a} + \uprho^\mathbf{b} = 1$. If a complement vector exists, it must be unique; the existence of a complement vector is equivalent to $\mathbf{a} \in E_{K,S}$, and a pair of complement exponent vectors correspond to a solution of the $S$-unit equation.

Suppose $q \in \Z$ is a prime number. We say $q$ \emph{avoids} $S$ if $q \not\in \mathfrak{p}$ for all ideals $\mathfrak{p} \in S$. If $q$ splits completely in $\OK$, then there are $d_K$ prime ideals above $q$ in $\OK$, say $\mathfrak{q}_0,\dots,\mathfrak{q}_{d_K-1}$. We let $\F_{\mathfrak{q}_j}$ denote the residue field of $\mathfrak{q}_j$. Since $q$ is completely split, we of course have $\F_{\mathfrak{q}_j} \cong \F_q$ for all $j$.

Suppose $\tau \in \OKS$, and $q$ is a rational prime number which splits completely in $\OK$ and which avoids $S$. The \emph{residue field vector} for $\tau$ (with respect to $q$) is
  \[ \rfv_q(\tau) := (\tau + \mathfrak{q}_0, \tau + \mathfrak{q}_1, \dots, \tau + \mathfrak{q}_{d_K-1}) \in \prod_{i=0}^{d_K-1} \F_{\mathfrak{q}_i}, \]
  where $\tau + \mathfrak{q}_j \in \F_{\mathfrak{q}_j}$ is the reduction of $\tau$ modulo $\mathfrak{q}_j$. The residue field vector depends on the ordering of the primes $\mathfrak{q}_j$ above $q$; we fix one ordering once and for all whenever we consider residue field vectors with respect to $q$.

Notice that we have the following commutative diagram, whose horizontal rows are exact.
\[
 \begin{tikzcd}
  & & & \prod_i \F_{\mathfrak{q}_i}^\times & \\
  1 \arrow[r] & \OKS^\times \cap (1 + q\OKS) \arrow[r] &
     \OKS^\times \arrow[r, "\rfv_q"] &
     \rfv_q(\OKS^\times) \arrow[r] \arrow[u, hook, "\mbox{\rotatebox{90}{$\subseteq$}}"] & 1 \\
  & (\OKS^\times)^{q-1} \arrow[u, hook, "\mbox{\rotatebox{90}{$\subseteq$}}"] &
     & & \\
  0 \arrow[r] & \{0 \} \times \bigl( (q-1)\Z \bigr)^t \arrow[u, hook, "\Phi_\uprho"] \arrow[r] &
     A_{K,S} \arrow[r, "\pi_{q-1}"] \arrow[uu, "\Phi_\uprho", "\cong"'] &
     A_{K,S,q-1} \arrow[r] \arrow[uu, dashed, "\exists"] & 0 \\
\end{tikzcd}
\]
  Suppose $\mathbf{a} \in A_{K,S,q-1}$. Since any two lifts $\mathbf{a}'$, $\mathbf{a}''$ of $\mathbf{a}$ to $A_{K,S}$ differ by a multiple of $(q-1)$, we see that $\Phi_\uprho(\mathbf{a}')$ and $\Phi_\uprho(\mathbf{a}'')$ differ by a perfect $(q-1)$th power, and so determine the same residue field vector. In other words, the dashed arrow in the diagram corresponds to a well-defined map $A_{K,S,q-1} \to \prod \F_{\mathfrak{q}_i}^\times$, and so the notion of a residue field vector for $\mathbf{a}$ is well-defined. With this in mind, we abuse notation slightly and also write $\rfv_q \mathbf{a} := \rfv_q \Phi_\brho (\mathbf{a}')$, where $\mathbf{a}' \in A_{K,S}$ is any lift of $\mathbf{a}$.
\begin{lemma}
  Suppose $\tau \in X_{K,S}$ and set $\eta = 1 - \tau$. Then
  \begin{enumerate}[itemsep=1pt, label=(\alph*)]
    \item $\rfv_q \tau + \rfv_q \eta = (1,1,\dots,1) \in \prod_i \F_{\mathfrak{q}_i}$.
    \item $\rfv_q \tau \in \prod_i \F_{\mathfrak{q}_i}^\times$.
    \item no entry of $\rfv_q \tau$ is $1$.
  \end{enumerate}
\end{lemma}
\begin{proof}
  Since $\tau + \eta = 1$, it follows that for any $j$, $\tau + \eta \equiv 1 \pmod{\mathfrak{q}_j}$, verifying (a). As $q$ avoids $S$, $\tau \not\in \mathfrak{q}_j$ for every $j$. This proves (b). Since (b) holds for both $\eta$ and $\tau$, (c) follows from (a).
\end{proof}

Suppose $\mathbf{a}$ is an exponent vector modulo $q-1$; i.e., $\mathbf{a} \in A_{K,S,q-1}$. We call $\mathbf{b} \in A_{K,S,q-1}$ a \emph{$(q-1)$-complement vector} for $\mathbf{a}$ if
\[ \rfv_q(\mathbf{a}) + \rfv_q(\mathbf{b}) = (1,1,\dots,1) \in \prod_i \F_{\mathfrak{q}_i}. \]
Existence of a $(q-1)$-complement vector is a necessary, but not sufficient, condition for $\mathbf{a}$ to lift to the exponent vector of a unit in a solution to the $S$-unit equation. Further, any particular $\mathbf{a}$ may have more than one $(q-1)$-complement vector associated to it. We set
\[ E_{K,S}(q-1) := \{ \mathbf{a} \in A_{K,S,q-1} : \mathbf{a} \text{ has a $(q-1)$-complement vector} \}. \]

\subsection{Execution of the sieve}
The strategy for the sieve is to play the sets $E_{K,S}(q-1)$ off of one another for multiple values of $q$. Choose a finite list $Q$ of rational prime numbers
\[Q = [q_0, q_1, \dots, q_{k-1}], \]
each of which splits completely in $K$ and avoids $S$, and such that
\[ \lcm(q_0-1, q_1-1, \dots, q_{k-1}-1) \geq 2B + 1. \]
Any true solution to the $S$-unit equation corresponds to exponent vectors found in the set $E_{K,S}$, and such vectors must reduce modulo $(q_j-1)$ to vectors in $E_{K,S}(q_j-1)$ for each $q_j \in Q$. Conversely, given a choice $\mathbf{a}_i \in E_{K,S}(q_i - 1)$ for each $0 \leq i < k$, there is at most one vector $\mathbf{a} \in A_{K,S}$ such that $\pi_{q_i-1}(\mathbf{a}) = \mathbf{a}_i$  for each $i$, while also satisfying $|\mathbf{a}| \leq B$. Define $\pi_Q$ to be the product of the maps $\pi_{q_i-1}$:
\[ \pi_Q \colon A_{K,S} \longrightarrow \prod_i A_{K,S,q_i-1}. \]
Certainly we have
\[ E_{K,S} \subseteq \pi_Q^{-1} \bigl( \prod_i E_{K,S}(q_i - 1) \bigr). \]
Because lifts from $\prod_i E_{K,S}(q_i -1)$ to $E_{K,S}$ are unique when they exist, $\prod_i E_{K,S}(q_i - 1)$ provides a reasonable proxy for the search space. We seek to replace each $E_{K,S}(q_i - 1)$ with a subset $Y_i \subseteq E_{K,S}(q_i - 1)$ such that we still have
\begin{equation}\label{eq:Ycond}
  E_{K,S} \subseteq \pi_Q^{-1} \bigl( \prod_i Y_i \bigr).
\end{equation}

Suppose $q_i, q_j$ are distinct primes in $Q$, and suppose $\mathbf{a}_i \in Y_i$, $\mathbf{a}_j \in Y_j$. We say $\mathbf{a}_i$ and $\mathbf{a}_j$ are \emph{compatible} if there exists $\mathbf{a} \in A_{K,S}$ such that $\pi_{q_i - 1}(\mathbf{a}) = \mathbf{a}_i$ and $\pi_{q_j - 1}(\mathbf{a}) = \mathbf{a}_j$. Notice that for any $i \neq j$, an element $\hat{\mathbf{a}} \in E_{K,S}$ reduces modulo $q_i - 1$ and $q_j - 1$ to produce a compatible pair of exponent vectors.

When $\mathbf{a}_i$ and $\mathbf{a}_j$ are compatible, we further call the pair \emph{complement compatible} if there exist $\mathbf{b}_i \in Y_i$ and $\mathbf{b}_j \in Y_j$ such that
\begin{itemize}
  \item $\mathbf{b}_i$ is $(q_i - 1)$-complementary to $\mathbf{a}_i$,
  \item $\mathbf{b}_j$ is $(q_j - 1)$-complementary to $\mathbf{a}_j$,
  \item $\mathbf{b}_i$ and $\mathbf{b}_j$ are compatible.
\end{itemize}

\begin{lemma}
Suppose the sets $Y_i \subseteq E_{K,S}(q_i - 1)$ satisfy condition \eqref{eq:Ycond}. Further, suppose $\mathbf{a}_i \in Y_i$, and set
  \[ Y'_j := \begin{cases} Y_j & j \neq i \\ Y_i - \{ \mathbf{a}_i \} & j = i \end{cases}. \]
  If there exists $j \neq i$ such that $Y_j$ contains no vectors which are complement compatible to $\mathbf{a}_i$, then
  \[ E_{K,S} \subseteq \pi_Q^{-1} \bigl( \prod_i Y_i' \bigr). \]
\end{lemma}
In other words, under the given condition, we will lose no true solutions by removing $\mathbf{a}_i$ from $Y_i$.
\begin{proof}
Towards a contradiction, suppose $\mathbf{a} \in E_{K,S}$ satisfies
\[ \pi_{q_i-1}(\mathbf{a}) = \mathbf{a}_i. \]
There is a unique $\mathbf{b} \in E_{K,S}$ satisfying $\Phi_\uprho(\mathbf{a}) + \Phi_\uprho(\mathbf{b}) = 1$. Set
\[ \mathbf{a}_j = \pi_{q_j-1}(\mathbf{a}), \quad \mathbf{b}_i = \pi_{q_i-1}(\mathbf{b}), \quad \mathbf{b}_j = \pi_{q_j-1}(\mathbf{b}). \]
Then $\mathbf{a}_i$ and $\mathbf{a}_j$ are compatible by definition. But since $\mathbf{a}_i$ and $\mathbf{a}_j$ cannot be complement compatible, the vectors $\mathbf{b}_i$ and $\mathbf{b}_j$ cannot be compatible. This is impossible, since $\mathbf{b} \in E_{K,S}$. Thus, no such $\mathbf{a}$ exists and the claim holds.
\end{proof}

The algorithm based on this lemma is the following.
\begin{algorithm}[Sieve]
Assume that $K$, $S$ are fixed and a representation of $\OKS^\times$ has been computed.
\begin{description}
  \item[\textbf{INPUT}] $Q = [q_0, q_1, \dots, q_{k-1}]$
  \item[\textbf{OUTPUT}] $Y_0, Y_1, \dots, Y_{k-1}$ satisfying \eqref{eq:Ycond}.
\end{description}
\begin{enumerate}[label={\bf \arabic*.}]
  \item Set $Y_i \longleftarrow E_{K,S}(q_i-1)$ for each $i$.
  \item Loop over $i \in \{0, 1, \dots, k-1 \}$:
     \begin{enumerate}[label=(\alph*)]
        \item Loop over $\mathbf{a}_i \in Y_i$:
        \begin{enumerate}[label=\roman*.]
          \item If $Y_i$ contains no $(q_i - 1)$-complement vector
                for $\mathbf{a}_i$, remove $\mathbf{a}_i$ from $Y_i$
          \item Loop over $j \in \{0, 1, \dots, i-1, i+1, \dots, k-1 \}$:
          \begin{itemize}[label=$\bullet$]
             \item If there are no $\mathbf{a}_j \in Y_j$ which are complement
                   compatible with $\mathbf{a}_i$, then remove $\mathbf{a}_i$ from $Y_i$.
          \end{itemize}
        \end{enumerate}
      \end{enumerate}
  \item Did Step $2$ remove any elements from any set $Y_i$?
  \begin{itemize}
    \item If \textbf{YES}, return to Step \textbf{2}.
    \item If \textbf{NO}, then \textbf{STOP}.
  \end{itemize}
\end{enumerate}
\end{algorithm}
Once the sieve has been completed, we may find all solutions to the $S$-unit equation by doing an exhaustive search over $\pi_Q^{-1}(\prod_i Y_i)$.

\section{Experimental Observations and Computational Choices}\label{sec:observ}
In developing this code and in pursuit of applications, we have computed a very large number of examples. Some observations and discussion may be enlightening to a reader who wishes to solve the $S$-unit equation for their own application.

Our implementation provides the function {\tt sieve\_below\_bound(K, S, B)}, which returns all solutions to the $S$-unit equation in $\OKS^\times$ up to a specified bound $B$ (the maximum absolute value of an entry in an exponent vector). This may be useful in settings where an exhaustive list of solutions is not needed. For example, in the field $K_g$ with ${g(x) = x^3 - 3x + 1}$, and $S_\mathrm{fin} = \{\mathfrak{p} : \mathfrak{p} \mid 2 \}$, the provable LLL-reduced exponent bound is $101$. However, all solutions actually satisfy the exponent bound $5$, and the command {\tt sieve\_below\_bound(K, S, 5)} executes in under $2$ seconds.

\subsection{Sieving vs. simple exhaustion}
Once a bound has been reduced as much as possible by LLL, this search space must be somehow exhausted. This general problem can be solved in multiple ways.  Those appearing in the literature can be generally described by the following three ideas:
\begin{enumerate}
  \item simple (non-number theory-based) exhaustion,
  \item sieve by reducing the problem modulo primes not in $S$, and
  \item sieve by reducing the problem modulo powers of primes in $S$.
\end{enumerate}
Idea (1) could be looked at through the more general lens of efficient programming, and a good programmer may be able to develop their own code to exhaust the search space effectively.  The current implementation uses idea (2), inspired by Smart's earlier exposition in \cite{Smart:1995} and is described here in Section \ref{sec:sieve}.  Item (3) paraphrases an interesting idea which is first due to de Weger in the case $K=\mathbb{Q}$ \cite{deWeger:1987} and which was generalized to arbitrary number fields $K$ for $S$-unit equations arising from Thue and Thue-Mahler equations by Tzanakis and de Weger \cite{Tzanakis-deWeger:1989, Tzanakis-deWeger:1991, Tzanakis-deWeger:1992}.  Wildanger \cite{Wildanger:2000} and Smart \cite{Smart:1999} worked out the details of the full generalization, which was later simplified by Evertse and Gy{\H o}ry \cite{Evertse-Gyory:2015}. This is an extremely promising and potentially effective method of reducing the search space, and has been implemented recently in special cases by several people, including Koutsianas \cite{Koutsianas:2017}, Bennett, Gherga, and Rechnitzer \cite{Bennett:2018}, von K\"{a}nel and Matschke \cite{vonKanel-Matschke:2016}, and others.  Future work will certainly focus on including this sieving technique for our functions.

In all these methods, we begin with the same search space as in (1), and the computational complexity of a brute force search is easy to estimate. Let $B$ be a bound for the maximum absolute value of an exponent in a solution to the $S$-unit equation. Since we are searching for a pair $(\tau_1, \tau_2)\in \left(\OKS^{\times}\right)^2$, the size of our search space is given by
\[ |A_{K,S}|^2 = w^2 (2B+1)^{2t}.\]
Thus a na\"{\i}ve brute force search has complexity $O\left(w^2(2B+1)^{2t}\right)$.  In practice, a simple exhaustive search can be carried out by checking, for each element $\tau_1$ of $A_{K,S}$, whether $1-\tau_1$ is an $S$-unit.  Assuming this check has constant time for a fixed $K$ and $S$, we get the less extreme complexity of $O\left(w(2B+1)^t\right)$.

In carrying out computations, we find that the resources required to sieve a search space vary greatly, even for number fields of the same degree and $S$-unit groups of the same rank. For example, we give the run time for three fields $K_g$, where $S_f$ is the set of primes above $3$ in $K_g$, in Table \ref{table:examples}. The column $N$ gives the total number of distinct solutions found. In each case, the LLL-reduced bound is below $40$, so complete sets of solutions are found in each case. Computations were performed in a paid account on the CoCalc platform in late 2018.
\begin{table}[h!]
\caption{Runtimes for {\tt sieve\_below\_bound(K, S, 40)}}
\label{table:examples}
\begin{tabular}{ccccc}
\toprule
$g(x)$ & $t$ & $w$ & $N$ & Runtime (in seconds) \\
\midrule
$x^4 - x^2 + 1$ & $2$ & $12$ & 16 & \phantom{0}1.16 \\
$x^4 + 9$ & $2$ & $4$ & 0 & \phantom{0}2.06 \\
$x^4 + 12x^2 + 18$ & $2$ & $2$ & 0 & 64.\phantom{00} \\
\bottomrule
\end{tabular}
\end{table}

The resources required depend on the size of the search space, but also can vary greatly based on the particular list of primes $Q$ chosen for the sieve, and even the order of those primes! In many cases, the sieve greatly reduces the time required to exhaust the space. In others, a brute force search of the reduced search space can actually be a better choice, as the sieving computation can take a mysteriously long time. Finding a way to understand and predict these difficulties is a priority for future work.  The implementation of idea (3) could also make this unnecessary. In all cases, it is worthwhile to find the smallest reduced bound possible, whether as input for the built-in sieve or for use in a brute force search.

\subsection{Finite place vs.~infinite place bounds}

In general, we find that the LLL-reduced bounds corresponding to $\mathfrak{p}_{\ell}$ infinite are smaller than the bounds for $\mathfrak{p}_{\ell}$ finite. To illustrate this, let $\mathscr{K}$ be the set of $85$ number fields $K$ satisfying
\[ 1 \leq [K:\Q] \leq 5, \qquad \Delta_K = \pm 2^a 3^b. \]
If $N \in \Z$, we set
\[ \begin{split}
S_{\mathrm{fin},K,N} & := \{ \mathfrak{p} \subseteq \OK : \mathfrak{p} \mid N \}, \\
S_{K,N} & := S_{\mathrm{fin}, K, N} \cup S_\infty \end{split} \]

For any choice of $K \in \mathscr{K}$ and $S = S_{K,N}$ where $N \in \{2, 3, 6 \}$, we have computed the LLL-reduced bounds under the assumption that $\mathfrak{p}_\ell$ is finite and under the assumption that $\mathfrak{p}_\ell$ is infinite. Complete bound data is available by email request to authors Malmskog or Rasmussen. Here we will consider only the case $S = S_{K,2}$. Now, let $B_1(K)$ and $B_2(K)$ be the bounds obtained in \S\ref{sec:LLL_red} under the assumption that $\mathfrak{p}_\ell$ is a finite or infinite place, respectively. In Figure \ref{fig:boundchart}, we plot both $B_1(K)$ and $B_2(K)$ against the root discriminant of $K$ (which ranges from $1.74$ to $26.56$ in $\mathscr{K}$.) The bound $B_1(K)$ usually exceeds $B_2(K)$, on average by a factor of $\approx 3.00$.

\begin{figure}[ht!]
  \includegraphics[width=.99\textwidth]{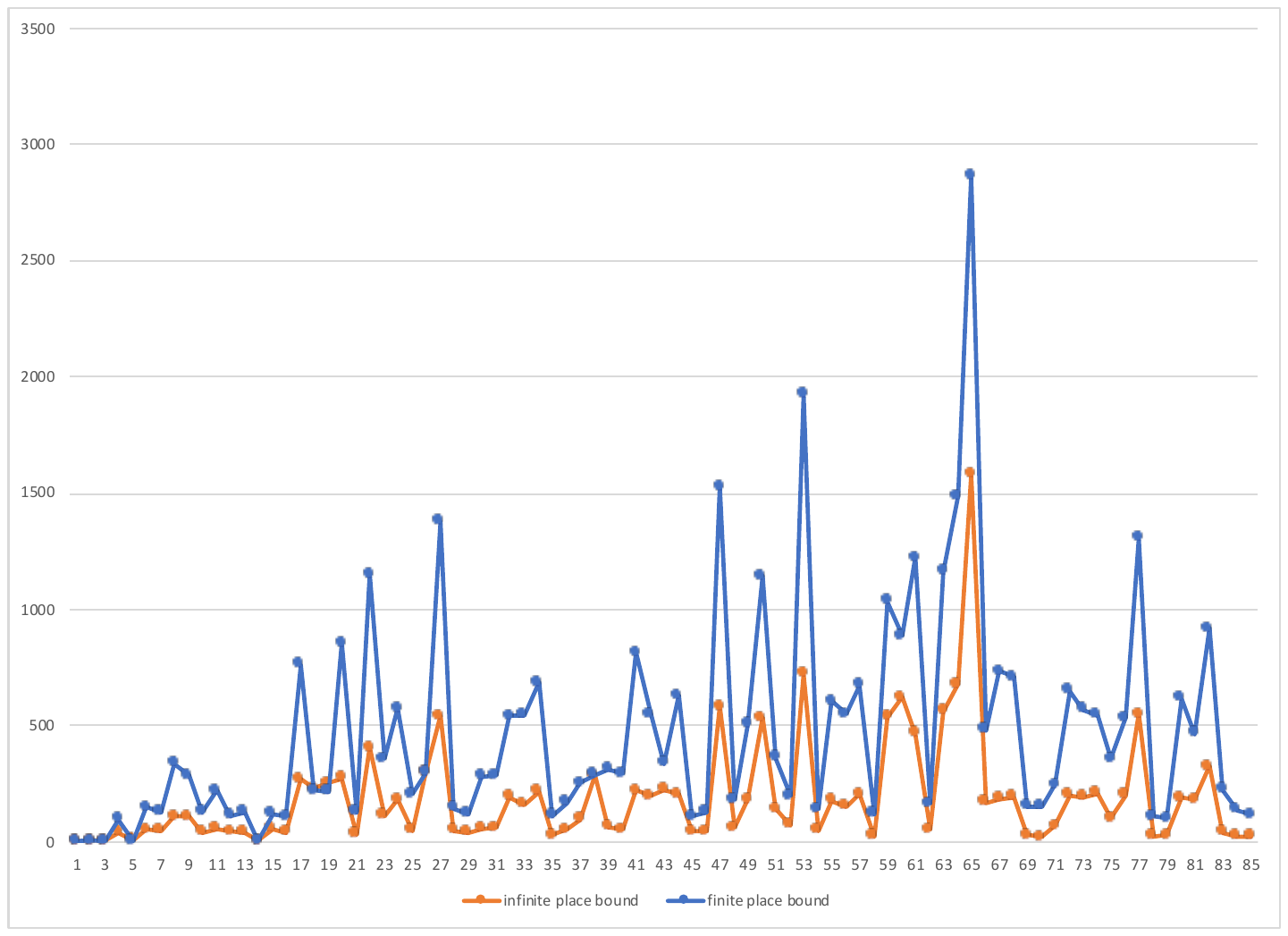}
  \caption{Bounds $B_1(K)$, $B_2(K)$ (vertical axis) for $S = S_{K,2}$ and $K \in \mathscr{K}$, plotted in order of increasing $\Delta_K^{1/d_K}$}
  \label{fig:boundchart}
\end{figure}

Because the disparity between these bounds is so large, we would prefer to use $B_2(K)$. Generally, we have no control over whether $\mathfrak{p}_\ell$ is finite or infinite. However, if $S$ contains only one finite place, a small trick allows us to use $B_2(K)$. If $(\tau_1,\tau_2)\in \OKS^{\times}$ is a solution to the $S$-unit equation, note that $(\frac{1}{\tau_1}, \frac{-\tau_2}{\tau_1})$ and $(\frac{1}{\tau_2}, \frac{-\tau_1}{\tau_2})$ are also $S$-unit equation solutions.  We define the \textit{solution cycle} of $\tau_1$ to be
\[C(\tau_1):=\left\{\tau_1, 1-\tau_1, \frac{1}{\tau_1}, 1-\frac{1}{\tau_1}, \frac{1}{1-\tau_1}, 1-\frac{1}{1-\tau_1}\right\}.\]
The following result is a restatement of \cite[Lemma 6.3]{Malmskog-Rasmussen:2016}.

\begin{lemma}\label{lemma:infinite_place}
Let $K$ be a number field, and suppose $S$ is a finite set of places of $K$ containing all infinite places and at most one finite place, (i.e. $\left| S_\mathrm{fin} \right| = 1$). Let $(\tau_1, \tau_2)$ be a solution to the $S$-unit equation over $K$. Then at least one element of $C(\tau_1)$ belongs to a solution with $\mathfrak{p}_\ell$ corresponding to an infinite place.
\end{lemma}

This implies that under the hypothesis of the lemma, some representative of each solution cycle has an exponent vector bounded by $B_2(K)$; recovering the entire solution cycle from one representative is trivial. Thus, we can determine all solutions to the $S$-unit equation.

It may seem that the hypothesis of Lemma \ref{lemma:infinite_place} -- that there is only one finite place in $S$ -- is a rather specialized condition. However, many interesting arithmetic applications involve searching for objects with ``good'' behavior away from one prime $p$. In such cases, we take $S = S_{K,p}$. Should $p$ ramify in $K$, the condition $|S_\mathrm{fin}| = 1$ is equivalent to $p$ being totally ramified, and this is not so uncommon when $[K:\Q]$ is small. Here, with $S = S_{K,2}$, the lemma applies for $72$ of the $85$ number fields in $\mathscr{K}$.

To illustrate the utility of Lemma \ref{lemma:infinite_place}, consider the ratio of the sizes of the search spaces for two bounds $B_1(K)$ and $B_2(K)$, given by
\[ R(K) = \frac{w^2 (2 B_1(K) + 1)^{2t}}{w^2 (2 B_2(K) + 1)^{2t}} \approx \left(\frac{B_1(K)}{B_2(K)} \right)^{2t}.\]
This quantifies the potential savings when the better bound may be used. For $S = S_{K,2}$, Figure \ref{fig:ratiochart} plots the savings $R(K)$ against the root discriminant of $K$ for the $72$ fields $K$ in $\mathscr{K}$ for which $\left| S_\mathrm{fin}\right| = 1$.
\begin{figure}[ht!]
  \includegraphics[width = .99\textwidth]{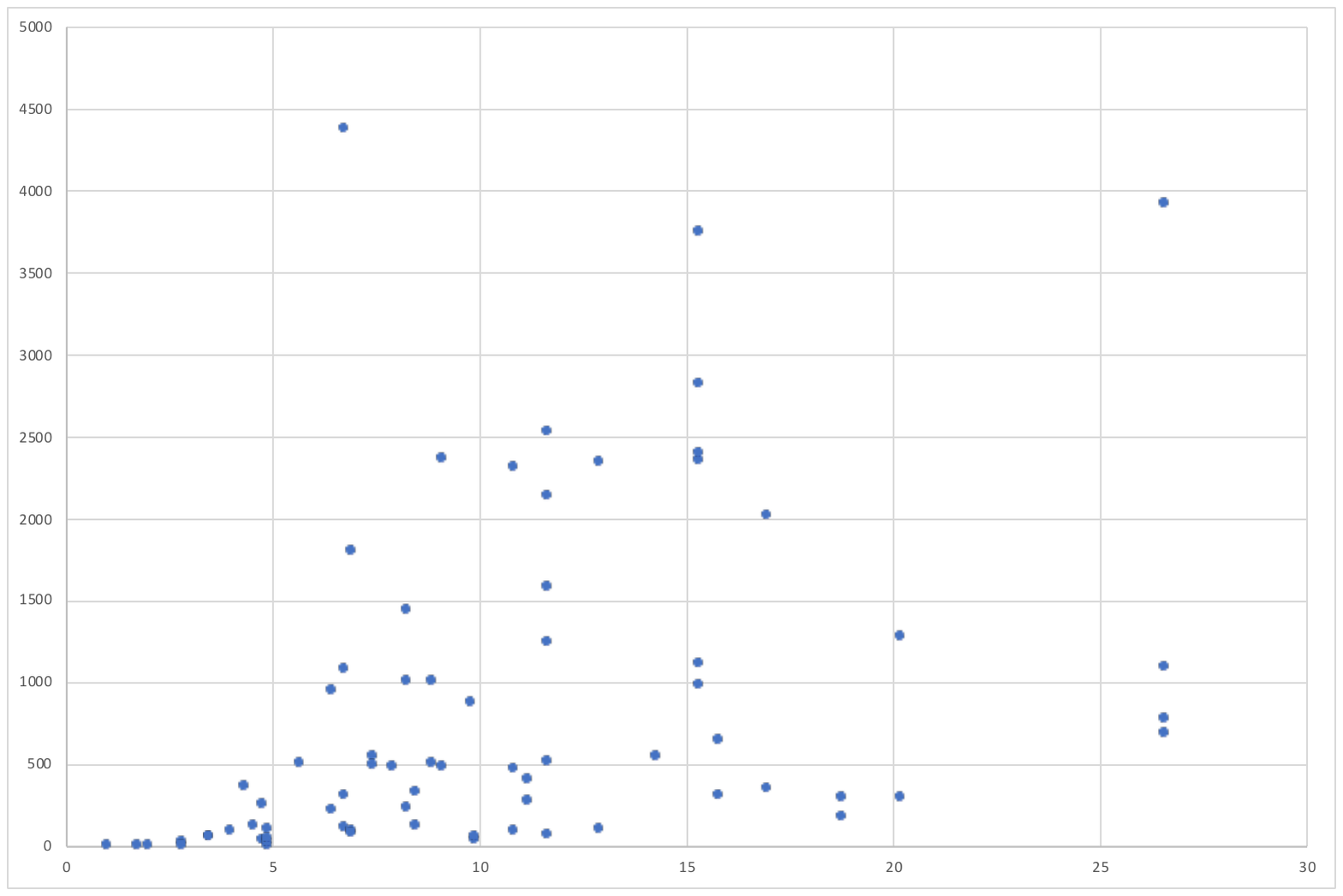}
  \caption{$R(K)$ (vertical axis) versus $\Delta_K^{1/d_K}$ (horizontal axis) for $K \in \mathscr{K}$, $S = S_{K,2}$, and $\left| S_\mathrm{fin} \right| = 1$.}
  \label{fig:ratiochart}
\end{figure}

\section{Applications}\label{sec:applications}

A major application of solving $S$-unit equations is in enumerating solutions to Shafarevich-type problems, for example finding complete lists of curves of a given type with particular reduction properties. The blueprint for this implementation came from Smart's 1997 enumeration of all genus 2 curves over $\mathbb{Q}$ with good reduction away from $p=2$ \cite{Smart:1997}, building off earlier work with Merriman \cite{Merriman-Smart:1993}.  In 2017, Malmskog and Rasmussen used these methods to determine all Picard curves defined over $\mathbb{Q}$ with good reduction away from $p=3$ \cite{Malmskog-Rasmussen:2016}. The same year, Koutsianas produced a new algorithm that uses solutions to the $S$-unit equation to find all elliptic curves over an arbitrary number field having good reduction outside $S$ \cite{Koutsianas:2017}. In the remainder of this article, we provide some new applications of the implementation.

\subsection{Asymptotic Fermat}

Let $K/\Q$ be a number field. We consider the nontrivial solutions $(a, b, c) \in K^3$ to the Fermat equation:
\[ \mathcal{C}_p: a^p + b^p + c^p = 0, \qquad \qquad abc \neq 0, \quad \text{$p > 3$ a prime}. \]
For fixed $p$, it follows from the work of Faltings that $\mathcal{C}_p(K)$ is finite, but it is reasonable to ask whether $\bigcup_p \mathcal{C}_p(K)$ is finite or infinite. Finiteness is equivalent to the condition that $\mathcal{C}_p(K) = \varnothing$ for sufficiently large $p$. We say $K$ satisfies \emph{asymptotic Fermat} if there exists a bound $B_K$ such that $p > B_K$ implies $\mathcal{C}_p(K) = \varnothing$.

There are several number fields $K$ known to satisfy asymptotic Fermat: Jarvis-Meekin \cite{Jarvis-Meekin:2004} demonstrate that $K = \Q(\sqrt{2})$ satisfies asymptotic Fermat with $B_K = 4$. Freitas-Siksek give an explicit family of real quadratic fields of density $\geq \frac{5}{6}$ which satisfy asymptotic Fermat. They also report that the real quartic field, $K = \Q(\sqrt{2 + \sqrt{2}})$ satisfies asymptotic Fermat.

In \cite{Freitas-Siksek:2015}, Freitas and Siksek find a condition on a totally real field $K$ which guarantees that $K$ satisfies asymptotic Fermat. For the remainder, suppose $K$ is totally real. Define
\begin{align*}
  S & = \{ \mathfrak{p} : \mathfrak{p} \text{ a nonzero prime ideal of } \OK \text{ which divides }2 \}, \\
  T & = \{ \mathfrak{p} \in S : f_{\mathfrak{p}} = 1 \}.
\end{align*}
\begin{theorem}[Freitas-Siksek]\label{thm:FS}
  Let $K/\Q$ be a totally real number field, with either $[K:\Q]$ odd or $T$ nonempty. Suppose that for every solution $(\tau_1, \tau_2)$ to the $S$-unit equation, there is some $\mathfrak{p} \in T$ such that $\max \{ |\ord_\mathfrak{p}(\tau_1)|, |\ord_\mathfrak{p}(\tau_2)|\} \leq 4 \ord_\mathfrak{p}(2)$. Then $K$ satisfies asymptotic Fermat.
\end{theorem}

\begin{remark}
  We note that Freitas-Siksek's result is actually stronger, and they provide additional conditions under which $K$ must satisfy asymptotic Fermat. Also, more recent work of \c{S}eng\"{u}n-Siksek \cite{Sengun-Siksek:2018} provides similar criteria for arbitrary number fields. However, the above formulation is sufficient for our application.
\end{remark}
The reader may recall that Wiles's classic proof of Fermat's Last Theorem proceeds by taking a hypothetical solution $(a,b,c)$ and noting that the associated Frey elliptic curve is forced to satisfy an impossible set of  constraints (that the curve is not modular). Freitas and Siksek's approach is similar. Given a solution to $\mathcal{C}_p$ over $K$, they produce an elliptic curve $E/K$ (related to, but distinct from, the Frey curve) whose $j$-invariant is arithmetically constrained. However, the $j$-invariant is determined by the $\lambda$-invariants of $E$; these $\lambda$ are guaranteed to arise as solutions to the $S$-unit equation over $K$. The result above follows from a delicate analysis of how these constraints interact.

We report a new list of cubic number fields $K/\Q$ which satisfy asymptotic Fermat. Using the implementation of the algorithm described in this paper, we find all solutions to the $S$-unit equation ($S$ as above), and verify the condition of Freitas-Siksek (this last step is trivial once all solutions have been determined).

Let $\mathscr{K}_X$ denote the set of totally real cubic number fields in which $2$ is totally ramified and which have absolute discriminant $\Delta_K$ satisfying $|\Delta_K| \leq X$. Table \ref{table:our_nfs} lists all the fields of $\mathscr{K}_X$ for $X = 2000$. For each $K \in \mathscr{K}_X$, we solved the appropriate $S$-unit equation, and by applying Theorem \ref{thm:FS}, verified that $K$ satisfies asymptotic Fermat. Our results are not effective, as Theorem \ref{thm:FS} does not provide the bound $B_K$.

For each $K \in \mathscr{K}_{2000}$, $f_K$ denotes a minimal polynomial for $K/\Q$; $\Delta_K$ is the absolute discriminant of $K$. Because $2$ is totally ramified, Lemma \ref{lemma:infinite_place} guarantees that every solution cycle will contain a solution with the extremal place $\mathfrak{p}_\ell$ infinite. Consequently, each solution cycle will contain at least one solution $(\tau_1, \tau_2)$ satisfying
\[ \tau_i = \uprho^{\mathbf{a}_i}, \qquad \left| \mathbf{a}_i \right| \leq K_1^\mathrm{LLL}. \]
(Finding the remaining solutions in the solution cycle is trivial even if they do not satisfy this bound.)

Finally, $N(S,K)$ indicates the number of distinct solutions $(\tau_1, \tau_2)$ to the $S$-unit equation found. (These are \emph{unordered} solutions, so that $(\tau_1, \tau_2)$ and $(\tau_2, \tau_1)$ are not considered distinct.) The reader should note that the two trivial solutions over $\Q$, $(-1,2)$ and $(\frac{1}{2}, \frac{1}{2})$, are counted in each field $K$.

\begin{table}[t!]\label{table:our_nfs}
  \caption{Fields in $\mathscr{K}_{2000}$ and number of $S$-unit equation solutions}
  \begin{tabular}{llcc}
    \toprule
    \multicolumn{1}{c}{$f_K$} & \multicolumn{1}{c}{$\Delta_K$} & $K_1^{\mathrm{LLL}}$ & $N(S,K)$  \\
    \midrule
    $x^3 - x^2 - 3x + 1$  & $2^2 \cdot 37$           & $225$ & $53$      \\
    \midrule
    $x^3 - x^2 - 5x - 1$  & $2^2 \cdot 101$          & $175$ & $11$     \\
    \midrule
    $x^3 - x^2 - 5x + 3$  & $2^2 \cdot 3 \cdot 47$   & $156$ & $5$      \\
    \midrule
    $x^3 - 6x - 2$        & $2^2 \cdot 3^3 \cdot 7$  & $161$ & $5$      \\
    \midrule
    $x^3 - x^2 - 7x - 3$  & $2^2 \cdot 197$          & $156$ & $8$      \\
    \midrule
    $x^3 - 8x - 6$        & $2^2 \cdot 269$          & $176$ & $8$     \\
    \midrule
    $x^3 - 10x - 10$      & $2^2 \cdot 5^2 \cdot 13$ & $156$ & $8$     \\
    \midrule
    $x^3 - x^2 - 7x + 5$  & $2^2 \cdot 349$          & $199$ & $8$     \\
    \midrule
    $x^3 - x^2 - 9x - 5$  & $2^2 \cdot 373$          & $162$ & $8$      \\
    \midrule
    $x^3 - x^2 - 7x + 1$  & $2^2 \cdot 3 \cdot 127$  & $180$ & $2$       \\
    \midrule
    $x^3 - x^2 - 9x + 11$ & $2^2 \cdot 389$          & $198$ & $8$   \\
    \midrule
    $x^3 - 12x - 14$      & $2^2 \cdot 3^4 \cdot 5$  & $164$ & $2$   \\
    \midrule
    $x^3 - 8x - 2$        & $2^2 \cdot 5 \cdot 97$   & $176$ & $5$   \\
    \bottomrule
  \end{tabular}
\end{table}

\subsection{Cubic Ramanujan-Nagell equations}

In 1913, Ramanujan conjectured that the only solutions of the Diophantine equation $x^2 + 7=2^n$ over the natural numbers satisfy $x \in \{1, 3, 5, 11, 181\}$ \cite{Ramanujan:1913}. This was settled in 1948 by Nagell \cite{Nagell:1961}. The more general family of equations,
\[
  Ax^2 + B = C^n, \qquad \qquad A, B, C \in \Z
\]
are called \emph{Ramanujan-Nagell equations}, and the literature for solving such equations is very rich (see for example \cite{Bugeaud-Shorey:2001,Bugeaud-Mignotte-Sikse:2006,Cohn:1993,Bennett-Skinner:2004}). Very recently \emph{cubic} Ramanujan-Nagell equations, have attracted the attention of mathematicians \cite{Bauer-Bennett:2018}. These are equations of the form
\[ f(x) = C^n, \qquad f(x) \in \Z[x], C \in \Z. \]
We consider the particular example
\begin{equation}\label{eq:cubicRN}
x^3 + 3^k = q^n, \qquad q > 3\ \text{prime}, \quad n, k > 0.
\end{equation}
If $q=2$, a more general version of \eqref{eq:cubicRN} is solved in \cite{Bauer-Bennett:2018}. Here, we prove the following theorem.
\begin{theorem}\label{thm:RamanujanNagell}
Let $q$ be a prime with $3 < q \leq \RNb$. All integer solutions of the cubic Ramanujan-Nagell equation \eqref{eq:cubicRN} with $k,n > 0$ are listed in Table \ref{table:cubicRN}.
\end{theorem}
Our method also works for the equation $x^3 + p^k = q^n$, where $p,q$ are different odd primes, and the proof is similar to the case $p=3$.

\begin{table}[!ht]\label{table:cubicRN}
  \caption{Solutions to \eqref{eq:cubicRN} with $3 < q \leq 500$.}
  \begin{tabular}{rrrrcrrrr}
    \toprule
    $q$ & $x$ & $k$ & $n$ & \null \quad \null & $q$ & $x$ & $k$ & $n$ \\
    \midrule
    $11$ & $2$ & $1$ & $1$ && $73$ & $4$ & $2$ & $1$ \\
    \midrule
    $17$ & $-4$ & $4$ & $1$ && $89$ & $2$ & $4$ & $1$ \\
    \midrule
    $17$ & $2$ & $2$ & $1$ && $179$ & $-4$ & $5$ & $1$ \\
    \midrule
    $19$ & $-2$ & $3$ & $1$ && $251$ & $2$ & $5$ & $1$ \\
    \midrule
    $67$ & $4$ & $1$ & $1$ && $307$ & $4$ & $5$ & $1$ \\
    \midrule
    $73$ & $-2$ & $4$ & $1$ \\
    \bottomrule
  \end{tabular}
\end{table}

\begin{proof}
  Let $K$ be the splitting field for $f(x) = x^3 + 3$. We observe $K$ is unramified outside $\{3,\infty\}$. In fact, $K$ has class number $1$ and is totally ramified at $3$. Let $\mathfrak{p} = \pi \OK$ be the unique prime in $K$ above $3$. Let $S$ be the set of all places of $K$ above $3$, $q$, or $\infty$.

  Suppose $(q,x,k,n)$ is a solution to \eqref{eq:cubicRN}. Let $\beta$ be a root of $f(x)$, and let $\zeta$ denote a primitive cube root of unity. Define
  \[ \alpha_i := (x + \zeta^i \beta^k), \quad \mathfrak{a}_i := \alpha_i \OK, \quad 0 \leq i \leq 2. \]
  Then we must have $\alpha_0 \alpha_1 \alpha_2 = q^n$ and $\mathfrak{a}_0 \mathfrak{a}_1 \mathfrak{a}_2 = q^n\OK$. For $i \neq j$,
  \[ \alpha_i - \alpha_j  = \zeta^i (1-\zeta^{i-j}) \beta^k \in \mathfrak{p}^{2k+3}. \]
Since $(3,q) = 1$, we see $\ord_\mathfrak{p} \alpha_i = 0$ for each $i$. Also, it follows that the $\mathfrak{a}_i$ are pairwise coprime. Thus, if $\mathfrak{q} \mid q\OK$, then exactly one $\mathfrak{a}_i$ is divisible by $\mathfrak{q}$, and $\ord_\mathfrak{q} \alpha_i = n$. Now fix $i' \in \{0,1,2\}$ so that $\ord_\mathfrak{q} \alpha_{i'} = n$ for at least one $\mathfrak{q} \mid q$. Choose $j' \neq i'$ and set
\[ \tau_1 := \frac{\alpha_{i'}}{\alpha_{i'} - \alpha_{j'}}, \quad \tau_2 := \frac{-\alpha_{j'}}{\alpha_{i'} - \alpha_{j'}}. \]
Then $(\tau_1, \tau_2)$ is a solution to the $S$-unit equation and $\ord_\mathfrak{q} \tau_1 = n$ for some $\mathfrak{q} \mid q$. Choose a root of unity $\rho_0$ and a basis $\rho_1, \dots, \rho_t$ for the torsion-free part of $\OKS^\times$. Choose $b_{i,j} \in \Z$ such that
\[ \tau_i = \prod_{j=0}^t \rho_j^{b_{i,j}}. \]
There exists $B$ such that $|b_{i,j}| \leq B$. Define
\[ c_3 := \sum_{j=1}^t \left| \ord_\mathfrak{p} \rho_j \right|, \qquad c_\mathfrak{q} := \sum_{j=1}^t \left| \ord_\mathfrak{q} \rho_j \right|, \]
and set $c_q := \max \{c_\mathfrak{q} : \mathfrak{q} \mid q \}$. By design,
\[ |2k+3| \leq c_3 B, \qquad |n| \leq c_q. \]
With these bounds established, the solutions to \eqref{eq:cubicRN} may now be determined by exhaustion.
\end{proof}
As a final remark, we observe that we may choose $\brho$ so that $c_3 = c_q = 1$. Let $\mathfrak{q}_1,\dots,\mathfrak{q}_g$ be the prime ideals in $K$ above $q$. As $\OK$ is a PID, we may choose $\lambda_i \in \OK$ such that $\mathfrak{q}_i = \lambda_i \OK$. Let $\xi_1, \xi_2$ generate the torsion-free part of $\OK^\times$. The choice $\brho = [\rho_0, \xi_1, \xi_2, \pi, \lambda_1, \dots, \lambda_g]$ now gives $c_3 = c_q = 1$.



\begin{thebibliography}{10}

\bibitem{Baker:1967}
A.~Baker.
\newblock Linear forms in the logarithms of algebraic numbers. {I}, {II},
  {III}.
\newblock {\em Mathematika 13 (1966), 204-216; ibid. 14 (1967), 102-107;
  ibid.}, 14:220--228, 1967.

\bibitem{BakerDavenport:1968}
A.~Baker and H.~Davenport.
\newblock  The equations $3x^2-2 = y^2$ and $8x^2-7 = z^2$.
\newblock {\em Quart. J. Math. Oxford}, 20(2):129-137, 1969.

\bibitem{Baker-Wustholz:1993}
A.~Baker and G.~W\"{u}stholz.
\newblock Logarithmic forms and group varieties.
\newblock {\em J. Reine Angew. Math.}, 442:19--62, 1993.

\bibitem{Baker-Wustholz:2007}
A.~Baker and G.~W\"{u}stholz.
\newblock {\em Logarithmic forms and {D}iophantine geometry}, volume~9 of {\em
  New Mathematical Monographs}.
\newblock Cambridge University Press, Cambridge, 2007.

\bibitem{Bauer-Bennett:2018}
M.~Bauer and M.~A.~Bennett.
\newblock Ramanujan-{N}agell cubics.
\newblock {\em Rocky Mountain J. Math.}, 48(2):385--412, 2018.

\bibitem{Bennett:2018}
M.~A.~Bennett, A.~Gherga, and A.~Rechnitzer.
\newblock Computing elliptic curves over {$\mathbb{Q}$}.
\newblock {\em Math. Comp.}, 88(317):1341--1390, 2019.

\bibitem{Bennett-Skinner:2004}
M.~A.~Bennett and C.~M.~Skinner.
\newblock Ternary {D}iophantine equations via {G}alois representations and
  modular forms.
\newblock {\em Canad. J. Math.}, 56(1):23--54, 2004.

\bibitem{Brumer:1967}
A.~Brumer.
\newblock On the units of algebraic number fields.
\newblock {\em Mathematika}, 14:121--124, 1967.

\bibitem{Bugeaud-Mignotte-Sikse:2006}
Y.~Bugeaud, M.~Mignotte, and S.~Siksek.
\newblock Classical and modular approaches to exponential {D}iophantine
  equations. {II}. {T}he {L}ebesgue-{N}agell equation.
\newblock {\em Compos. Math.}, 142(1):31--62, 2006.

\bibitem{Bugeaud-Shorey:2001}
Y.~Bugeaud and T.~N. Shorey.
\newblock On the number of solutions of the generalized {R}amanujan-{N}agell
  equation.
\newblock {\em J. Reine Angew. Math.}, 539:55--74, 2001.

\bibitem{Cohn:1993}
J.~H.~E. Cohn.
\newblock The {D}iophantine equation {$x^2+C=y^n$}.
\newblock {\em Acta Arith.}, 65(4):367--381, 1993.

\bibitem{deWeger:1987}
B.~M.~M. de~Weger.
\newblock Solving exponential {D}iophantine equations using lattice basis
  reduction algorithms.
\newblock {\em J.~Number Theory}, 26(3):325--367, 1987.

\bibitem{deWeger:thesis}
B.~M.~M. de~Weger.
\newblock {\em Algorithms for {D}iophantine {E}quations}.
\newblock PhD thesis, Universiteit Leiden, 1988.

\bibitem{deWeger:1989}
B.~M.~M. de~Weger.
\newblock {\em Algorithms for {D}iophantine equations}, volume~65 of {\em CWI
  Tract}.
\newblock Stichting Mathematisch Centrum, Centrum voor Wiskunde en Informatica,
  Amsterdam, 1989.

\bibitem{Evertse-Gyory:2015}
J.-H. Evertse and K.~Gy\H{o}ry.
\newblock {\em Unit equations in {D}iophantine number theory}, volume 146 of
  {\em Cambridge Studies in Advanced Mathematics}.
\newblock Cambridge University Press, Cambridge, 2015.

\bibitem{Freitas-Siksek:2015}
N.~Freitas and S.~Siksek.
\newblock The asymptotic {F}ermat's last theorem for five-sixths of real
  quadratic fields.
\newblock {\em Compos.~Math.}, 151(8):1395--1415, 2015.

\bibitem{Fuchs:2014}
C.~Fuchs.
\newblock {\em On some applications of {D}iophantine approximations}, volume~2
  of {\em Quaderni/Monographs}.
\newblock Edizioni della Normale, Pisa, 2014.
\newblock A translation of Carl Ludwig Siegel's ``\"{U}ber einige Anwendungen
  diophantischer Approximationen'' by Clemens Fuchs, With a commentary and the
  article ``Integral points on curves: Siegel's theorem after Siegel's proof''
  by Fuchs and Umberto Zannier, Edited by Zannier.

\bibitem{Gyory:1979}
K.~Gy{\H o}ry.
\newblock On the number of solutions of linear equations in units of an
  algebraic number field.
\newblock {\em Comment. Math. Helv.}, 54(4):583--600, 1979.

\bibitem{Gyory:2019}
K.~Gy{\H o}ry.
\newblock {Bounds for the solutions of $S$-unit equations and decomposable form
  equations II}.
\newblock arXiv:1901.11289, January 2019.

\bibitem{GyoryYu:2006}
K.~Gy{\H o}ry and K.~Yu.
\newblock {Bounds for the solutions of $S$-unit equations and decomposable form
  equations}.
\newblock {\em Acta Arithmetica}, 123(1):9--41, 2006.

\bibitem{Hasse:1980}
H.~Hasse.
\newblock {\em Number theory}.
\newblock Classics in Mathematics. Springer-Verlag, Berlin, german edition,
  2002.
\newblock Reprint of the 1980 English edition [Springer, Berlin; MR0562104
  (81c:12001b)], Edited and with a preface by Horst G{\"u}nter Zimmer.

\bibitem{Jarvis-Meekin:2004}
F.~Jarvis and P.~Meekin.
\newblock The {F}ermat equation over {$\mathbb{Q}(\sqrt{2})$}.
\newblock {\em J. Number Theory}, 109(1):182--196, 2004.

\bibitem{Koutsianas:2017}
A.~Koutsianas.
\newblock Computing all elliptic curves over an arbitrary number field with
  prescribed primes of bad reduction.
\newblock {\em Experimental Mathematics}, 2017.

\bibitem{LLL:1982}
A.~K. Lenstra, H.~W. Lenstra, Jr., and L.~Lov{\'a}sz.
\newblock Factoring polynomials with rational coefficients.
\newblock {\em Math. Ann.}, 261(4):515--534, 1982.

\bibitem{Mahler:1933}
K.~Mahler.
\newblock Zur {A}pproximation algebraischer {Z}ahlen. {I}.
\newblock {\em Math. Ann.}, 107(1):691--730, 1933.

\bibitem{Malmskog-Rasmussen:2016}
B.~Malmskog and C.~Rasmussen.
\newblock Picard curves over {$\mathbb{Q}$} with good reduction away from 3.
\newblock {\em LMS J. Comput. Math.}, 19(2):382--408, 2016.

\bibitem{Merriman-Smart:1993}
J.~R. Merriman and N.~P. Smart.
\newblock Curves of genus {$2$} with good reduction away from {$2$} with a
  rational {W}eierstrass point.
\newblock {\em Math. Proc. Cambridge Philos. Soc.}, 114(2):203--214, 1993.

\bibitem{Nagell:1961}
T.~Nagell.
\newblock The {D}iophantine equation {$x^{2}+7=2^{n}$}.
\newblock {\em Ark. Mat.}, 4:185--187, 1961.

\bibitem{Petho-deWeger:1986}
A.~Peth\"{o} and B.~M.~M. de~Weger.
\newblock Products of prime powers in binary recurrence sequences. {I}. {T}he
  hyperbolic case, with an application to the generalized {R}amanujan-{N}agell
  equation.
\newblock {\em Math. Comp.}, 47(176):713--727, 1986.

\bibitem{vonKanel-Matschke:2016}
B.~M. R.~von K{\"a}nel and Benjamin Matschke.
\newblock {Solving {$S$}-unit, Mordell, Thue, Thue-Mahler and generalized
  Ramanujan-Nagell equations via Shimura-Taniyama conjecture}.
\newblock preprint, arXiv:1605.06079, 2016.

\bibitem{Ramanujan:1913}
S.~Ramanujan.
\newblock Question \#464.
\newblock {\em J.~Indian Math.~Soc.}, 5:120, 1913.


\bibitem{SAGE}
{Sage Developers}.
\newblock {\em {S}ageMath, the {S}age {M}athematics {S}oftware {S}ystem
  ({V}ersion 8.4)}, 2018.

\bibitem{Sengun-Siksek:2018}
M.~H. {\c S}eng{\"u}n and S.~Siksek.
\newblock On the asymptotic {F}ermat's last theorem over number fields.
\newblock {\em Comment. Math. Helv.}, 93(2):359--375, 2018.

\bibitem{Siegel:1929}
C.~L. Siegel.
\newblock {\"U}ber einige anwendungen diophantischer approximationen.
\newblock {\em Abh. der Preuss. Akad. der Wissenschaften Phys. Math. Kl.},
  1:209--266, 1929.

\bibitem{Smart:1995}
N.~P. Smart.
\newblock The solution of triangularly connected decomposable form equations.
\newblock {\em Math. Comp.}, 64(210):819--840, 1995.

\bibitem{Smart:1997}
N.~P. Smart.
\newblock {$S$}-unit equations, binary forms and curves of genus {$2$}.
\newblock {\em Proc. London Math. Soc. (3)}, 75(2):271--307, 1997.

\bibitem{Smart:1998}
N.~P. Smart.
\newblock {\em The algorithmic resolution of {D}iophantine equations},
  volume~41 of {\em London Mathematical Society Student Texts}.
\newblock Cambridge University Press, Cambridge, 1998.

\bibitem{Smart:1999}
N.~P. Smart.
\newblock Determining the small solutions to {$S$}-unit equations.
\newblock {\em Math. Comp.}, 68(228):1687--1699, 1999.

\bibitem{Thue:1909}
A.~Thue.
\newblock {\"U}ber {A}nn{\"a}herungswerte algebraischer {Z}ahlen.
\newblock {\em J. Reine Angew. Math.}, 135:284--305, 1909.

\bibitem{Tzanakis-deWeger:1989}
N.~Tzanakis and B.~M.~M. de~Weger.
\newblock On the practical solution of the {T}hue equation.
\newblock {\em J. Number Theory}, 31(2):99--132, 1989.

\bibitem{Tzanakis-deWeger:1991}
N.~Tzanakis and B.~M.~M. de~Weger.
\newblock Solving a specific {T}hue-{M}ahler equation.
\newblock {\em Math. Comp.}, 57(196):799--815, 1991.

\bibitem{Tzanakis-deWeger:1992}
N.~Tzanakis and B.~M.~M. de~Weger.
\newblock How to explicitly solve a {T}hue-{M}ahler equation.
\newblock {\em Compositio Math.}, 84(3):223--288, 1992.

\bibitem{Wildanger:2000}
K.~Wildanger.
\newblock {\"U}ber das {L}{\"o}sen von {E}inheiten- und {I}ndexformgleichungen
  in algebraischen {Z}ahlk{\"o}rpern.
\newblock {\em J. Number Theory}, 82(2):188--224, 2000.

\bibitem{Yu:1989}
K.~Yu.
\newblock Linear forms in {$p$}-adic logarithms.
\newblock {\em Acta Arith.}, 53(2):107--186, 1989.

\bibitem{Yu:1990}
K.~Yu.
\newblock Linear forms in {$p$}-adic logarithms. {II}.
\newblock {\em Compositio Math.}, 74(1):15--113, 1990.

\bibitem{Yu:1994}
K.~Yu.
\newblock Linear forms in {$p$}-adic logarithms. {III}.
\newblock {\em Compositio Math.}, 91(3):241--276, 1994.

\bibitem{Yu:2007}
K.~Yu.
\newblock {$p$}-adic logarithmic forms and group varieties. {III}.
\newblock {\em Forum Math.}, 19(2):187--280, 2007.

\end{thebibliography}
\end{document}